\documentclass[11pt,a4paper]{article}

\usepackage[margin=2.7cm]{geometry}
\usepackage[T1]{fontenc}
\usepackage[utf8]{inputenc}
\usepackage{lmodern}
\usepackage{amsmath,amssymb,amsthm,mathtools}
\usepackage{microtype}
\usepackage{enumitem}
\usepackage{xcolor}
\usepackage[hidelinks]{hyperref}
\usepackage[nameinlink,capitalise]{cleveref}

\newtheorem{theorem}{Theorem}[section]
\newtheorem{lemma}[theorem]{Lemma}
\newtheorem{proposition}[theorem]{Proposition}
\newtheorem{corollary}[theorem]{Corollary}
\theoremstyle{definition}

\newtheorem{remark}[theorem]{Remark}

\newcommand{\C}{\mathbb C}

\newcommand{\T}{\mathbb T}
\newcommand{\supp}{\operatorname{supp}}

\newcommand{\tr}{\operatorname{tr}}

\newcommand{\op}{\mathrm{op}}

\newcommand{\BH}{\mathrm{BH}}

\title{Intrinsic Bohnenblust--Hille Inequalities for Local Qudit Systems}
\author{Andreas Defant \and Daniel Galicer}
\date{}

\begin{document}

\maketitle

\begin{abstract}
We study dimension-free Bohnenblust--Hille inequalities for local operators on
systems of $K$-level qudits. 
For the operator space of support at most $d$, uniformly over all tensor-product orthonormal operator bases, we prove a Bohnenblust--Hille inequality with the optimal exponent $2d/(d+1)$ whose asymptotic exponential base in the interaction order is of order $O(\sqrt K)$.
Our proof is intrinsic, based on a one-site
scalarization and block-transversal decoupling, and does not rely on a
scalarization to the cyclic group or on a Remez-type argument. We complement the upper bound by showing that the asymptotic exponential
Bohnenblust--Hille base satisfies
\(
cK^{1/4}\le \beta(K)\le C\sqrt K
\)
with universal constants \(c,C>0\).
In this sense, the present work may be viewed as continuing the line of work
of Slote--Volberg--Zhang from a complementary intrinsic viewpoint. We also
revisit their Gell--Mann and Heisenberg--Weyl scalarization procedures. In the
Heisenberg--Weyl setting, the prime-dimensional case already yields the
optimal interaction exponent, whereas for composite $K$ the scalar total
degree used in their reduction leads to a larger exponent. 
A support-sensitive formulation recovers the optimal interaction exponent for every $K$, while the intrinsic argument gives the stronger dependence on the local dimension at the level of the asymptotic exponential base.
As further consequences, we determine the sharp scale, up to factors exponential in $d$ with base depending only on $K$, of coefficient $\ell_1$-normalization and
unconditionality on the exact-support spaces. This yields dimension-free
coefficient sparsification, including sparse generalized-Pauli approximation
in Heisenberg--Weyl coordinates, as well as normalization and query bounds for
canonical LCU/qubitization constructions.
\end{abstract}

\tableofcontents

\section{Introduction}\label{sec:introduction}

Bohnenblust--Hille inequalities control coefficient norms of bounded
low-complexity functions by the uniform norm, with constants independent of
the ambient dimension. On the Boolean cube, the optimal exponent for degree
\(d\) is \(2d/(d+1)\), and the corresponding constants grow subexponentially
in \(d\); see \cite{DMPBoolean} and, for a recent refinement in the cyclic
setting, \cite{DGMMM-support,PR}. For the historical background of the
Bohnenblust--Hille inequality, and in particular for the renewed interest in
the subject initiated by \cite{DFOSO} and further developed in
\cite{BPS}, we refer to \cite{DGMS}.

Quantum versions of this principle were introduced in connection with the analysis and
learning of local observables \cite{HCP,VZnoncomm}. For systems of qudits,
a sequence of works by Becker, Klein, Slote, Volberg, and Zhang developed
dimension-free estimates through reductions of Gell--Mann and
Heisenberg--Weyl expansions to commutative Boolean and cyclic models; see
\cite{BSVZnormdesign,KSVZ,SVZcyclic,SVZqudit,VZnoncomm}.
Related recent
developments include \cite{Slote2026,SloteVolberg2026,SVZtightness}.

The noncommutative setting exhibits a genuinely different behavior in the
interaction order. For qubits, Volberg--Zhang proved an exponential upper
bound, which was subsequently improved by Becker--Slote--Volberg--Zhang to
an exponential base of \(\sqrt{3}\), up to the subexponential Boolean
Bohnenblust--Hille factor. On the other hand, a recent result of Slote shows
that exponential growth in \(d\) is unavoidable already for \(K=2\).
Thus, for general local dimension \(K\), the natural question is not whether
exponential growth can be avoided, but rather how its exponential base
depends on \(K\).

In this paper we study this question for the intrinsic $d$-local operator
space, without fixing a preferred local coordinate system. We write
$M_K^0=\{A\in M_K(\mathbb C):\operatorname{tr}A=0\}$ and denote by
$\mathcal H_{\le d,K}(n)$ the subspace of $M_K(\mathbb C)^{\otimes n}$
consisting of operators whose tensor expansion involves at most $d$
nonidentity sites. We also write $\mathcal H_{d,K}(n)$ for the corresponding
exact-support space, consisting of operators whose tensor expansion involves
exactly $d$ nonidentity sites. If
$\mathcal B=\{I,B_1,\ldots,B_{K^2-1}\}$ is any orthonormal basis for the
normalized Hilbert--Schmidt inner product, its tensor powers give a coefficient
sequence $\widehat A_{\mathcal B}$. 

Our main result is a dimension-free Bohnenblust--Hille inequality at the
optimal exponent $2d/(d+1)$, uniform in the number of sites $n$ and in
the choice of the tensor-product Hilbert--Schmidt orthonormal basis.
Namely,
$$
\|\widehat A_{\mathcal B}\|_{\ell_{p_d}}
\le
K(C\sqrt K)^d
\|A\|_{\rm op}
\qquad
A\in\mathcal H_{\leq d,K}(n), 
$$
where $C>0$ is an universal constant. For the exact-support spaces
$\mathcal H_{d,K}(n)$, let $C_{d,K}$ denote the optimal dimension-free
constant and set
\[
 \beta(K):=\limsup_{d\to\infty} C_{d,K}^{1/d}.
\]
Our upper estimate gives $\beta(K)\le C\sqrt K$, while a complementary
Slote-type construction yields
\[
 cK^{1/4}\le \beta(K)\le C\sqrt K
\]
with universal constants $c,C>0$. Thus the dependence of the exponential
base on the local dimension is polynomially constrained from both sides.

The proof of the upper bound is intrinsic to the decomposition
$M_K=\mathbb CI\oplus M_K^0$. Its basic ingredients are a one-site
scalarization, whose relevant norm grows only like $\sqrt K$, and a
block-transversal decoupling of the physical sites. On the transversal part,
the scalarization tensorizes and can be combined with the Blei inequality and
the complex Khinchin--Steinhaus inequality. In particular, the argument does
not pass through a scalarization to the cyclic group and does not require a
Remez-type inequality. 
This is what allows us to keep the optimal interaction exponent while obtaining an $O(\sqrt K)$ bound for the asymptotic exponential base.

The present work may be viewed as continuing the line of research of
Slote--Volberg--Zhang from a complementary intrinsic viewpoint. We therefore
also revisit their Gell--Mann and Heisenberg--Weyl scalarization procedures
\cite{SVZqudit}. Their arguments provide natural benchmarks for the general
qudit problem. 
In the Heisenberg--Weyl setting, the prime-dimensional case
already gives the optimal interaction exponent. For composite $K$, however,
the scalar total degree used in their reduction may increase from $d$ to
$(K-1)d$, and consequently the resulting Bohnenblust--Hille exponent is
larger than $2d/(d+1)$. A support-sensitive formulation recovers the optimal interaction exponent for every $K$, while the intrinsic argument gives the stronger dependence on the local dimension at the level of the asymptotic exponential base.

There is also a useful basis-independence principle behind this comparison.
At the Bohnenblust--Hille exponent, changing between two normalized local
Hilbert--Schmidt orthonormal bases entails only a one-site loss, independent
of the interaction order at the level of the exponential base. Consequently,
bounds obtained in the Gell--Mann or Heisenberg--Weyl coordinates can be
transferred to arbitrary tensor-product operator bases. Together with the
support-sensitive cyclic Bohnenblust--Hille inequality of
\cite{DGMMM-support} and the recent subexponential estimate of
Pellegrino--Raposo \cite{PR}, this also clarifies what can be obtained from the
existing scalarization procedures in composite dimensions. 
These arguments are useful for comparison, while the direct intrinsic proof gives the substantially stronger $O(\sqrt K)$ bound for the asymptotic exponential base.

Beyond the Bohnenblust--Hille inequality itself, we study two related
coefficient-geometric quantities. For the exact $d$-support space, the
coefficient $\ell_1$ normalization and the unconditional basis constant have,
up to factors depending exponentially only on $d$ and $K$, the sharp scale
$(n/d)^{(d-1)/2}$. The lower bound follows from a basis-adapted commutative
corner and the corresponding Hamming-scheme estimates from
\cite{DGMMM-support}. This coefficient geometry also has two natural quantum
consequences. First, bounded local operators admit dimension-free sparse
tensor-product approximations in normalized Hilbert--Schmidt norm,
including sparse generalized-Pauli approximations in
Heisenberg--Weyl coordinates. Second, in unitary
tensor-product coordinates the coefficient $\ell_1$ norm is exactly the
normalization parameter of the canonical linear-combination-of-unitaries
representation, leading to PREPARE/SELECT block encodings and the corresponding
qubitization bounds \cite{ChildsWiebe,LowChuang}. These are statements about
the canonical unitary-basis encoding, rather than lower bounds for arbitrary
Hamiltonian-simulation algorithms.

The paper is organized as follows. Section~2 introduces the commutative and
operator notions of support and degree. Section~3 proves the intrinsic
Bohnenblust--Hille inequality, first for exact support and then for the full
$d$-local filtration, and establishes the lower bound for the exponential
base. Section~4 discusses the Gell--Mann and Heisenberg--Weyl bases, the
change-of-coordinates principle, and the relation with the scalarization
procedures of \cite{SVZqudit}. Section~5 develops the coefficient
$\ell_1$ geometry and records its consequences for coefficient sparsification
and canonical LCU/qubitization constructions.
\section{Commutative and noncommutative low-complexity spaces}

This section fixes the two notions of complexity that will be used below.  Support
degree counts active coordinates and is intrinsic to the product decomposition;
total degree also records the local frequencies carried by those coordinates.  We
first make this distinction on $C_K^n$, then pass to $M_K(\mathbb C)^{\otimes n}$.
\subsection{Commutative support and total-degree spaces}
\label{sec:Commutative support and total-degree spaces}

Fix $K\ge2$, let $\zeta=e^{2\pi i/K}$ and
$C_K=\{1,\zeta,\ldots,\zeta^{K-1}\}$, and identify $\widehat{C_K^n}$ with
$\mathbb Z_K^n$.  For $m=(m_1,\ldots,m_n)\in\mathbb Z_K^n$, write
$\chi_m(x)=x_1^{m_1}\cdots x_n^{m_n}$.  We use
$s(m)=|\supp(m)|$, with $\supp(m)=\{j:m_j\ne0\}$, for the support degree and
$|m|=m_1+\cdots+m_n$ for the ordinary total degree.

The exact and cumulative support spaces are
\[
 B_d(C_K^n):=\operatorname{span}\{\chi_m:s(m)=d\},
 \qquad
 B_{\le d}(C_K^n):=\operatorname{span}\{\chi_m:s(m)\le d\}
 =\bigoplus_{r=0}^d B_r(C_K^n),
\]
whereas the corresponding total-degree spaces are
\[
 P_d(C_K^n):=\operatorname{span}\{\chi_m:|m|=d\},
 \qquad
 P_{\le d}(C_K^n):=\operatorname{span}\{\chi_m:|m|\le d\}
 =\bigoplus_{r=0}^dP_r(C_K^n).
\]
These are subspaces of $C(C_K^n)$ with the supremum norm.  Since every nonzero
coordinate contributes at least one unit to $|m|$, one has
$s(m)\le |m|$ and therefore
\[
 P_{\le d}(C_K^n)\subset B_{\le d}(C_K^n).
\]
At exact total degree, several support levels may occur:
\begin{equation}\label{eq:P-support-decomposition}
 P_d(C_K^n)
 =
 \bigoplus_{r=0}^d\bigl(P_d(C_K^n)\cap B_r(C_K^n)\bigr).
\end{equation}
For $K=2$ support and total degree coincide; for $K\ge3$ they are genuinely
different.

The support decomposition can also be written without characters.  Let
$\mathbb E_K$ denote averaging over $C_K$, and let $C_0(C_K)$ be its kernel.
Then $C(C_K)=\mathbb C1\oplus C_0(C_K)$, and the mean-zero space is spanned
by the nontrivial characters.  For $S\subset[n]$ set
\[
 B_S(C_K^n)
 :=
 \left(\bigotimes_{j\in S}C_0(C_K)\right)
 \otimes
 \left(\bigotimes_{j\notin S}\mathbb C1\right).
\]
Then $B_d(C_K^n)=\bigoplus_{|S|=d}B_S(C_K^n)$.  Thus an inactive coordinate
contributes a constant, while an active coordinate contributes a mean-zero local
function.  This is the model for the noncommutative definition below.

\subsection{Exact $d$-support and $d$-local operator spaces}

The canonical local decomposition is
$M_K(\mathbb C)=\mathbb CI\oplus M_K^0$, where
\[
 M_K^0=\{A\in M_K(\mathbb C):\tr A=0\}.
\]
For $S\subset[n]$ define
\[
 \mathcal H_S
 :=
 \left(\bigotimes_{j\in S}M_K^0\right)
 \otimes
 \left(\bigotimes_{j\notin S}\mathbb CI\right)
 \subset M_K(\mathbb C)^{\otimes n}.
\]
The \emph{exact $d$-support operator space} and the \emph{$d$-local operator space} are
\[
 \mathcal H_{d,K}(n):=\bigoplus_{|S|=d}\mathcal H_S,
 \qquad
 \mathcal H_{\le d,K}(n):=\bigoplus_{|S|\le d}\mathcal H_S
 =\bigoplus_{r=0}^d\mathcal H_{r,K}(n).
\]

They carry the ambient operator norm.  The spaces
\(\mathcal H_{d,K}(n)\) and \(\mathcal H_{\le d,K}(n)\) are intrinsic:
they are determined by the canonical decomposition
\(
M_K(\mathbb C)=\mathbb CI\oplus M_K^0
\)
and do not depend on a choice of operator basis.  

To describe them in
coordinates, we now fix a normalized Hilbert--Schmidt orthonormal local
basis.
On \(M_K(\mathbb C)\) we use the normalized Hilbert--Schmidt inner
product
\[
 \langle A,B\rangle_{2,K}
 :=
 \frac1K\operatorname{tr}(AB^*).
\]
Let
\[
 \mathcal B=\{B_0,B_1,\ldots,B_{K^2-1}\}
 =
 \{I,B_1,\ldots,B_{K^2-1}\}
\]
be an orthonormal basis for \(\langle\cdot,\cdot\rangle_{2,K}\), with
\(B_\mu\in M_K^0\) for \(1\le\mu\le K^2-1\).
On the \(n\)-site algebra
\[
 M_K(\mathbb C)^{\otimes n}\cong M_{K^n}(\mathbb C)
\]
we use the corresponding normalized Hilbert--Schmidt inner product
\[
 \langle A,B\rangle_{2,K^n}
 :=
 \frac1{K^n}\operatorname{tr}(AB^*).
\]
For elementary tensors,
\[
 \left\langle
 A_1\otimes\cdots\otimes A_n,\,
 B_1\otimes\cdots\otimes B_n
 \right\rangle_{2,K^n}
 =
 \prod_{j=1}^n
 \langle A_j,B_j\rangle_{2,K}.
\]
We write
\[
 B_\beta
 :=
 B_{\beta_1}\otimes\cdots\otimes B_{\beta_n}
 \,, \qquad 
 \beta=(\beta_1,\ldots,\beta_n)
 \in\{0,\ldots,K^2-1\}^n
\]
Then
\[
 \mathcal B^{\otimes n}
 :=
 \{B_\beta:\beta\in\{0,\ldots,K^2-1\}^n\}
\]
is an orthonormal basis of \(M_K(\mathbb C)^{\otimes n}\) for
\(\langle\cdot,\cdot\rangle_{2,K^n}\).  We refer to it as the
tensor-product operator basis associated with \(\mathcal B\).  For
\(A\in M_K(\mathbb C)^{\otimes n}\), we define its
\(\mathcal B^{\otimes n}\)-coefficients by
\[
 \widehat A_{\mathcal B}(\beta)
 :=
 \langle A,B_\beta\rangle_{2,K^n}
 =
 \frac1{K^n}\operatorname{tr}(AB_\beta^*),
\]
so that
\[
 A
 =
 \sum_{\beta}
 \widehat A_{\mathcal B}(\beta)B_\beta.
\]
Put
\[
 \operatorname{supp}(\beta)
 :=
 \{j\in[n]:\beta_j\ne0\}.
\]
Since \(B_0=I\) and \(B_\mu\in M_K^0\) for \(\mu\ne0\), the intrinsic
support spaces admit the equivalent coordinate descriptions
\[
 \mathcal H_{d,K}(n)
 =
 \operatorname{span}
 \bigl\{
 B_\beta:\ |\operatorname{supp}(\beta)|=d
 \bigr\},
\]
and
\[
 \mathcal H_{\le d,K}(n)
 =
 \operatorname{span}
 \bigl\{
 B_\beta:\ |\operatorname{supp}(\beta)|\le d
 \bigr\}.
\]
Thus the basis \(\mathcal B\) merely provides coordinates for the
intrinsic support decomposition.  

This basis-independent notion of
support has a natural commutative model, to which we now turn.
We shall repeatedly use a canonical commutative corner.  Let
$E_{rs}=|e_r\rangle\langle e_s|$ and
$\mathcal D_K=\operatorname{span}\{E_{00},\ldots,E_{K-1,K-1}\}$.  For
$m\in\mathbb Z_K$ put
\begin{equation}\label{eq:Dm-definition}
 D_m:=\sum_{r=0}^{K-1}\zeta^{mr}E_{rr}.
\end{equation}
Then $D_0=I$ and $\tr D_m=0$ for $m\ne0$, so
$(D_m)_{m\ne0}$ is a basis of $\mathcal D_K^0:=\mathcal D_K\cap M_K^0$.
Define $J$ on characters by
\[
 J(\chi_m):=D_{m_1}\otimes\cdots\otimes D_{m_n}.
\]

\begin{lemma}[Commutative spherical corners]
\label{lem:spherical-corners}
For every $0\le d\le n$, the map $J$ restricts to isometric embeddings
\[
 J:B_d(C_K^n)\longrightarrow\mathcal H_{d,K}(n),
 \qquad
 J:B_{\le d}(C_K^n)\longrightarrow\mathcal H_{\le d,K}(n),
\]
and both ranges are $1$-complemented.
\end{lemma}

\begin{proof}
Because $D_0=I$ and $D_m\in M_K^0$ for $m\ne0$, the tensor $J(\chi_m)$ lies in
$\mathcal H_{\supp(m)}$; hence $J$ preserves support level.  If
$f=\sum_m a_m\chi_m$, then all matrices $J(\chi_m)$ are diagonal and
\[
 \|J(f)\|_{\rm op}
 =
 \max_{r\in\mathbb Z_K^n}
 \left|\sum_m a_m\zeta^{\langle m,r\rangle}\right|
 =
 \|f\|_{L^\infty(C_K^n)}.
\]
Thus $J$ is isometric.

Let $\mathbb E_{\mathcal D}(A)=\sum_{r=0}^{K-1}E_{rr}AE_{rr}$ be the canonical
conditional expectation onto $\mathcal D_K$.  It is unital and completely
positive, so $\|\mathbb E_{\mathcal D}\|=1$; it also fixes $I$, preserves trace,
and maps $M_K^0$ onto $\mathcal D_K^0$.  Therefore
$P_{\mathcal D}:=\mathbb E_{\mathcal D}^{\otimes n}$ is a norm-one projection
preserving every $\mathcal H_S$, with
\begin{equation*}
     P_{\mathcal D}(\mathcal H_{d,K}(n))=J(B_d(C_K^n)),
 \qquad
 P_{\mathcal D}(\mathcal H_{\le d,K}(n))=J(B_{\le d}(C_K^n)).\qedhere
\end{equation*}
\end{proof}

Thus $B_d(C_K^n)$ is a contractively complemented commutative model inside the
exact $d$-support operator space, and $B_{\le d}(C_K^n)$ plays the same role inside
the $d$-local operator space.

\subsection{Gell--Mann and Heisenberg--Weyl bases}

We record two standard single-site orthonormal systems that will be used as reference
coordinates.  On the standard basis of $\mathbb C^K$, let
$X|r\rangle=|r+1\rangle$ and $Z|r\rangle=\zeta^r|r\rangle$, with indices modulo
$K$, so $ZX=\zeta XZ$.  For $(a,b)\in\mathbb Z_K^2$ set
\[
 W_{a,b}:=X^aZ^b.
\]
The Heisenberg--Weyl family $\{W_{a,b}:(a,b)\in\mathbb Z_K^2\}$ is orthonormal
for $\langle A,B\rangle=K^{-1}\tr(A^*B)$; moreover
$\tr W_{a,b}=0$ unless $(a,b)=(0,0)$.  Hence the nonidentity Weyl matrices form
an orthonormal basis of $M_K^0$.

For
$\alpha=((a_1,b_1),\ldots,(a_n,b_n))\in(\mathbb Z_K^2)^n$, write
$W_\alpha=W_{a_1,b_1}\otimes\cdots\otimes W_{a_n,b_n}$ and
$s_{\rm HW}(\alpha)=|\{j:(a_j,b_j)\ne(0,0)\}|$.  Then
\[
 \mathcal H_{d,K}(n)=\operatorname{span}\{W_\alpha:s_{\rm HW}(\alpha)=d\},
 \qquad
 \mathcal H_{\le d,K}(n)=\operatorname{span}\{W_\alpha:s_{\rm HW}(\alpha)\le d\}.
\]
For the generalized Gell--Mann basis, define for $0\le j<k\le K-1$
\[
 A_{jk}=\sqrt{\frac K2}(E_{jk}+E_{kj}),
 \qquad
 B_{jk}=\sqrt{\frac K2}(-iE_{jk}+iE_{kj}),
\]
and, for $1\le m\le K-1$,
\[
 C_m=\sqrt{\frac K{m^2+m}}
 \left(\sum_{r=0}^{m-1}E_{rr}-mE_{mm}\right).
\]
Then
\[
 \operatorname{GM}(K)
 =
 \{I\}\cup\{A_{jk},B_{jk}:0\le j<k\le K-1\}
 \cup\{C_m:1\le m\le K-1\}
\]
is again orthonormal, and every nonidentity element is traceless.  If
$\mathfrak G_K=\{0\}\cup\mathfrak G_K^\times$ indexes this basis with $M_0=I$,
then for $\gamma\in\mathfrak G_K^n$ the tensors
$M_\gamma=M_{\gamma_1}\otimes\cdots\otimes M_{\gamma_n}$ satisfy
\[
 \mathcal H_{d,K}(n)=\operatorname{span}\{M_\gamma:|\supp(\gamma)|=d\},
 \qquad
 \mathcal H_{\le d,K}(n)=\operatorname{span}\{M_\gamma:|\supp(\gamma)|\le d\}.
\]
Thus both systems provide tensor-product coordinates for the same intrinsic
support spaces \(\mathcal H_{d,K}(n)\) and \(\mathcal H_{\le d,K}(n)\).

\subsection{The Heisenberg--Weyl total-degree filtration}
\label{sec:HW-total-degree}

The HW system carries an additional, basis-dependent total degree.  For a multi-index
$\alpha=((a_j,b_j))_{j=1}^n$, choose the representatives
$a_j,b_j\in\{0,\ldots,K-1\}$ and set
\[
 |\alpha|_{\rm HW}:=\sum_{j=1}^n(a_j+b_j).
\]
For $0\le d\le2(K-1)n$ define
\[
 \mathcal P^{\rm HW}_{d,K}(n):=\operatorname{span}\{W_\alpha:|\alpha|_{\rm HW}=d\},
 \qquad
 \mathcal P^{\rm HW}_{\le d,K}(n):=\operatorname{span}\{W_\alpha:|\alpha|_{\rm HW}\le d\}
 =\bigoplus_{r=0}^d\mathcal P^{\rm HW}_{r,K}(n).
\]
Every active site contributes at least one unit to $|\alpha|_{\rm HW}$, hence
\[
 \mathcal P^{\rm HW}_{\le d,K}(n)\subset\mathcal H_{\le d,K}(n).
\]
At exact total degree one has the support decomposition
\[
 \mathcal P^{\rm HW}_{d,K}(n)
 =
 \bigoplus_{r=0}^d
 \left(\mathcal P^{\rm HW}_{d,K}(n)\cap\mathcal H_{r,K}(n)\right),
\]
which is the operator analogue of \eqref{eq:P-support-decomposition}.

On the diagonal corner, commutative and HW total degrees agree.  Indeed,
$D_m=Z^m$, so $J(\chi_m)=Z^{m_1}\otimes\cdots\otimes Z^{m_n}$ and the associated
HW index $\iota(m)=((0,m_1),\ldots,(0,m_n))$ satisfies
\begin{equation}\label{eq:degree-diagonal-agreement}
 |\iota(m)|_{\rm HW}=|m|.
\end{equation}

\begin{lemma}[Commutative homogeneous corners]
\label{lem:homogeneous-corners}
For every $d\ge0$, the map $J$ restricts to isometric embeddings
\[
 J:P_d(C_K^n)\longrightarrow\mathcal P^{\rm HW}_{d,K}(n),
 \qquad
 J:P_{\le d}(C_K^n)\longrightarrow\mathcal P^{\rm HW}_{\le d,K}(n),
\]
and both ranges are $1$-complemented.
\end{lemma}

\begin{proof}
Isometry follows from Lemma~\ref{lem:spherical-corners}.  By
\eqref{eq:degree-diagonal-agreement}, $J$ sends total degree $d$ into HW total
degree $d$.  Moreover,
\[
 \mathbb E_{\mathcal D}(X^aZ^b)
 =
 \begin{cases}
 Z^b,&a=0,\\
 0,&a\ne0.
 \end{cases}
\]
Thus $P_{\mathcal D}=\mathbb E_{\mathcal D}^{\otimes n}$ preserves HW total
degree whenever a basis tensor survives the projection, and
\[
 P_{\mathcal D}(\mathcal P^{\rm HW}_{d,K}(n))=J(P_d(C_K^n)),
 \qquad
 P_{\mathcal D}(\mathcal P^{\rm HW}_{\le d,K}(n))=J(P_{\le d}(C_K^n)).
\]
Since $\|P_{\mathcal D}\|=1$, both ranges are $1$-complemented.
\end{proof}

The two filtrations therefore fit the parallel picture
\[
 \begin{gathered}
 B_d(C_K^n)\stackrel{1\text{-compl.}}{\hookrightarrow}\mathcal H_{d,K}(n),
 \qquad
 B_{\le d}(C_K^n)\stackrel{1\text{-compl.}}{\hookrightarrow}\mathcal H_{\le d,K}(n),\\
 P_d(C_K^n)\stackrel{1\text{-compl.}}{\hookrightarrow}\mathcal P^{\rm HW}_{d,K}(n),
 \qquad
 P_{\le d}(C_K^n)\stackrel{1\text{-compl.}}{\hookrightarrow}\mathcal P^{\rm HW}_{\le d,K}(n),
 \end{gathered}
\]
together with the inclusions $P_{\le d}(C_K^n)\subset B_{\le d}(C_K^n)$ and
$\mathcal P^{\rm HW}_{\le d,K}(n)\subset\mathcal H_{\le d,K}(n)$.  The remainder
of the paper is organized by support degree; HW total degree reappears only as a
corollary of the support-sensitive theory.

\section{Intrinsic noncommutative Bohnenblust--Hille inequality}
\label{sec:BH-theory}

We now turn to the main quantitative result.  In view of the exponential
lower bound of \cite{Slote2026}, the main issue is the dependence of the
dimension-free Bohnenblust--Hille constants on the local dimension \(K\).
Put
\[
 p_d=\frac{2d}{d+1}.
\]
For fixed \(K\) and \(d\), let \(C_{d,K}\) be the least constant, uniform in
\(n\) and in the choice of local basis \(\mathcal B\), such that
\[
 \|\widehat A_{\mathcal B}\|_{\ell_{p_d}}
 \le C_{d,K}\|A\|_{\rm op},
 \qquad
 A\in\mathcal H_{d,K}(n).
\]
Set
\[
 \beta(K)=\limsup_{d\to\infty}C_{d,K}^{1/d}.
\]
Finally, let \(a_1\) denote a universal constant in the complex
Khinchin--Steinhaus inequality at \(p=1\); see
\cite[Section~6.2 and Corollary~6.10]{DGMS}.
Our main estimate is the following.

\begin{theorem}[Exact-support Bohnenblust--Hille upper bound]
\label{thm:local-BH}
Let $K\ge2$, $d\ge2$, $n\ge d$, and let
$\mathcal B=\{I,B_1,\ldots,B_{K^2-1}\}$ be any normalized
Hilbert--Schmidt orthonormal local operator basis with
$B_\mu\in M_K^0$ for $\mu\ge1$.  Then every
$A\in\mathcal H_{d,K}(n)$ satisfies
\[
 \|\widehat A_{\mathcal B}\|_{\ell_{p_d}}
 \le
 \sqrt{K^2-1}\,a_1^{2(d-1)}
 (8\sqrt K)^d
 \left(\frac{d^d}{d!}\right)^{1/p_d}
 \|A\|_{\rm op}.
\]
Consequently,
\[
\beta(K)\le 8a_1^2\sqrt e\,\sqrt K,
\]
and in particular \(\beta(K)=O(\sqrt K)\).  
\end{theorem}

The exponent \(p_d=2d/(d+1)\) is optimal.  Indeed, by
Lemma~\ref{lem:spherical-corners} the corresponding commutative
support space embeds isometrically into the exact-support operator
space, and the exponent is optimal in the commutative setting;
see~\cite{DGMMM-support}.

The proof of Theorem~\ref{thm:local-BH} is developed in Subsections~\ref{Two-Steinhaus scalarization}--~\ref{Block-transversal decoupling} and assembled at the end of
Subsection~\ref{Block-transversal decoupling}.  Its central decoupling device is adapted from
\cite{DGMMM-support}: a random partition of the sites isolates
block-transversal pieces, so that scalarization is applied only blockwise
rather than to the full operator.  On each such piece, standard tools from
the theory of Bohnenblust--Hille inequalities, chiefly Blei's inequality and
the complex Khinchin--Steinhaus inequality, yield the required coefficient
estimate, while averaging over the partition recovers the full coefficient
vector.  No reduction to cyclic groups, Remez-type argument, or previously
established Bohnenblust--Hille inequality enters the proof.

\subsection{Two-Steinhaus scalarization}
\label{Two-Steinhaus scalarization}
 Write
\[
 M:=K^2-1,
 \qquad
 \langle B,C\rangle_2=\frac1K\tr(BC^*).
\]
For $z=(z_1,\ldots,z_K)\in\T^K$, put
\[
 u_z=\frac1{\sqrt K}(z_1,\ldots,z_K)\in\C^K.
\]
For $B\in M_K(\C)$ and $(z,w)\in\T^K\times\T^K$, define
\begin{equation}\label{eq:gB-section3}
 g_B(z,w)
 =\sqrt K\,\langle Bu_z,u_w\rangle
 =\frac1{\sqrt K}\sum_{r,s=1}^K B_{rs}z_s\overline{w_r}.
\end{equation}
Here and below the Hilbert-space inner product is linear in the first
variable.
If $E\subset[n]$ is a set of sites, we use the product probability space
\[
 \Omega_E=(\T^K\times\T^K)^E.
\]
A point of $\Omega_E$ is written
\[
 \omega=((z_i,w_i))_{i\in E}.
\]
For $i\in E$ and $H\in M_K(\C)$ we define the site-$i$ copy
\[
 g_H^{(i)}(\omega):=g_H(z_i,w_i).
\]
Thus $g_H^{(i)}$ depends only on the coordinate $(z_i,w_i)$ of the site $i$.
For a basis matrix we abbreviate
\[
 g_{i,\mu}:=g_{B_\mu}^{(i)}.
\]

\begin{lemma}[Exact $L_2$ normalization]
\label{lem:L2-isometry-section3}
For all $B,C\in M_K(\C)$,
\[
 \int_{\T^K\times\T^K}g_B(z,w)\overline{g_C(z,w)}\,dz\,dw
 =\langle B,C\rangle_2.
\]
Consequently, for $i,j\in E$ and $1\le\mu,\nu\le M$,
\begin{equation}\label{eq:site-orthogonality-section3}
 \int_{\Omega_E}g_{i,\mu}\overline{g_{j,\nu}}
 =\delta_{ij}\delta_{\mu\nu}.
\end{equation}
\end{lemma}

\begin{proof}
Expanding \eqref{eq:gB-section3},
\[
 g_B(z,w)\overline{g_C(z,w)}
 =\frac1K
 \sum_{r,s,r',s'}B_{rs}\overline{C_{r's'}}
 z_s\overline{z_{s'}}\,\overline{w_r}w_{r'}.
\]
The characters of the torus are orthogonal, so only the terms with
$r=r'$ and $s=s'$ survive.  Hence
\[
 \int g_B\overline{g_C}
 =\frac1K\sum_{r,s}B_{rs}\overline{C_{rs}}
 =\frac1K\tr(BC^*)
 =\langle B,C\rangle_2.
\]
If $i=j$, this gives
\[
 \int_{\Omega_E}g_{i,\mu}\overline{g_{i,\nu}}
 =\langle B_\mu,B_\nu\rangle_2
 =\delta_{\mu\nu}.
\]
If $i\ne j$, the two functions depend on independent coordinates and the
integral factors.  Moreover $\int g_B=0$ for every $B$, because every
Steinhaus coordinate has mean zero.  Thus the integral is zero.  This proves
\eqref{eq:site-orthogonality-section3}.
\end{proof}

\begin{lemma}[Independent phase alignment]
\label{lem:phase-alignment-section3}
Let $E$ be finite and let $H_i\in M_K(\C)$ for $i\in E$.  Then
\begin{equation}\label{eq:phase-alignment-section3}
 \left\|\sum_{i\in E}g_{H_i}^{(i)}\right\|_{L_\infty(\Omega_E)}
 =\sum_{i\in E}\|g_{H_i}\|_{L_\infty(\T^K\times\T^K)}.
\end{equation}
\end{lemma}

\begin{proof}
The inequality ``$\le$'' is the triangle inequality.  For the reverse
inequality, for every $i$ choose $(z_i,w_i)$ at which $|g_{H_i}|$ attains its
maximum, and write
\[
 g_{H_i}(z_i,w_i)=e^{\mathrm i\theta_i}\|g_{H_i}\|_\infty.
\]
Since
\[
 g_H(e^{\mathrm it}z,w)=e^{\mathrm it}g_H(z,w),
\]
replacing $z_i$ by $e^{-\mathrm i\theta_i}z_i$ makes the $i$th value equal to
$\|g_{H_i}\|_\infty$.  These rotations can be made independently because the
$i$th function depends only on $(z_i,w_i)$.  At the resulting point of
$\Omega_E$ all summands are nonnegative real numbers, and their sum is exactly
the right-hand side of~ \eqref{eq:phase-alignment-section3}.
\end{proof}

\subsection{Scalar Bohnenblust--Hille estimate}

We next recall the traditional one-index form of Blei's inequality.  We write
all indices explicitly because in our application each scalar index is a pair
consisting of a physical site and a local operator coordinate.

\begin{lemma}[Blei inequality]
\label{lem:traditional-blei-section3}
Let $E_1,\ldots,E_d$ be finite sets and let
\[
 a_{(i_1,\mu_1),\ldots,(i_d,\mu_d)},
 \qquad
 (i_j,\mu_j)\in E_j\times[M],
\]
be a scalar array.  For $1\le k\le d$, put
\begin{align*}
 M_k(a)
 :=\sum_{(i_k,\mu_k)\in E_k\times[M]}
 \Bigg(
 \sum_{\substack{(i_j,\mu_j)\in E_j\times[M]\\ j\ne k}}
 \left|
 a_{(i_1,\mu_1),\ldots,(i_d,\mu_d)}
 \right|^2
 \Bigg)^{1/2}.
\end{align*}
Then
\[
 \Bigg(
 \sum_{(i_1,\mu_1)\in E_1\times[M]}\cdots
 \sum_{(i_d,\mu_d)\in E_d\times[M]}
 \left|a_{(i_1,\mu_1),\ldots,(i_d,\mu_d)}\right|^{p_d}
 \Bigg)^{1/p_d}
 \le
 \prod_{k=1}^d M_k(a)^{1/d}.
\]
\end{lemma}

\begin{proof}
This is the standard Blei inequality; see
\cite[Proposition~6.11]{DGMS}.  If the sets $E_j\times[M]$ have different
cardinalities, pad the array with zeros and apply the equal-cardinality
statement.
\end{proof}

\begin{lemma}[The degree-one estimate]
\label{lem:degree-one-section3}
Let $E$ be finite and let $c_{i,\mu}\in\C$ for
$i\in E$ and $1\le\mu\le M$.  Put
\[
 H_i=\sum_{\mu=1}^M c_{i,\mu}B_\mu
 \qquad(i\in E)
\]
and
\[
 G(\omega)
 =\sum_{i\in E}\sum_{\mu=1}^M
 c_{i,\mu}g_{i,\mu}(\omega).
\]
Then
\[
 \sum_{i\in E}\sum_{\mu=1}^M|c_{i,\mu}|
 \le \sqrt M\,\|G\|_{L_\infty(\Omega_E)}.
\]
\end{lemma}

\begin{proof}
For each fixed $i$, Cauchy--Schwarz and the normalized Hilbert--Schmidt
orthonormality of the $B_\mu$ give
\[
 \sum_{\mu=1}^M|c_{i,\mu}|
 \le
 \sqrt M\left(\sum_{\mu=1}^M|c_{i,\mu}|^2\right)^{1/2}
 =\sqrt M\,\|H_i\|_2.
\]
By Lemma~\ref{lem:L2-isometry-section3},
$\|g_{H_i}\|_2=\|H_i\|_2$, and hence
$\|g_{H_i}\|_\infty\ge\|H_i\|_2$.  Therefore
\begin{align*}
 \sum_{i\in E}\sum_{\mu=1}^M|c_{i,\mu}|
 \le \sqrt M\sum_{i\in E}\|H_i\|_2 &\le \sqrt M\sum_{i\in E}\|g_{H_i}\|_\infty=\sqrt M\left\|\sum_{i\in E}g_{H_i}^{(i)}\right\|_\infty
 =\sqrt M\,\|G\|_\infty,
\end{align*}
where the equality is Lemma~\ref{lem:phase-alignment-section3}.
\end{proof}

\begin{proposition}[Scalar Bohnenblust--Hille estimate]
\label{prop:scalar-BH-section3}
Let $E_1,\ldots,E_d$ be finite.  For a scalar array
$a_{(i_1,\mu_1),\ldots,(i_d,\mu_d)}$, with
$(i_j,\mu_j)\in E_j\times[M]$, define
\begin{align*}
 F(\omega_1,\ldots,\omega_d)
 :=&\sum_{(i_1,\mu_1)\in E_1\times[M]}\cdots
 \sum_{(i_d,\mu_d)\in E_d\times[M]}
 a_{(i_1,\mu_1),\ldots,(i_d,\mu_d)}
 \prod_{j=1}^d g_{B_{\mu_j}}^{(i_j)}(\omega_j),
\end{align*}
where $\omega_j\in\Omega_{E_j}$.  Then
\begin{equation}\label{eq:scalar-BH-section3}
 \left(
 \sum_{(i_1,\mu_1)\in E_1\times[M]}\cdots
 \sum_{(i_d,\mu_d)\in E_d\times[M]}
 \left|a_{(i_1,\mu_1),\ldots,(i_d,\mu_d)}\right|^{p_d}
 \right)^{1/p_d}
 \le
 \sqrt M\,a_1^{2(d-1)}\|F\|_\infty.
\end{equation}
\end{proposition}

\begin{proof}
Fix $k\in\{1,\ldots,d\}$ and a pair
$(i_k,\mu_k)\in E_k\times[M]$.  Apply the usual complex
Khinchin--Steinhaus inequality to the two Steinhaus families in each of the
other $d-1$ blocks, iterating it with Minkowski.  This is precisely the
standard multilinear Khinchin--Steinhaus argument; see
\cite[Section~6.2 and Corollary~6.10]{DGMS}.  Using
Lemma~\ref{lem:L2-isometry-section3} in every block to identify the
$\ell_2$ norm of the operator-basis coefficients, we obtain
\begin{equation}\label{eq:KS-pairs-section3}
\begin{aligned}
 &\Bigg( \sum_{\substack{(i_j,\mu_j)\in E_j\times[M]\\ j\ne k}}
 \left|a_{(i_1,\mu_1),\ldots,(i_d,\mu_d)}\right|^2
 \Bigg)^{1/2}
 \\
 &\quad\le
 a_1^{2(d-1)}
 \int_{\prod_{j\ne k}\Omega_{E_j}}
 \Bigg|
 \sum_{\substack{(i_j,\mu_j)\in E_j\times[M]\\ j\ne k}}
 a_{(i_1,\mu_1),\ldots,(i_d,\mu_d)}
 \prod_{j\ne k}g_{B_{\mu_j}}^{(i_j)}(\omega_j)
 \Bigg|d\omega_{\widehat k}.
\end{aligned}
\end{equation}
Sum \eqref{eq:KS-pairs-section3} over
$(i_k,\mu_k)\in E_k\times[M]$.  For fixed
$\omega_{\widehat k}=(\omega_j)_{j\ne k}$ define
\begin{align*}
 c_{i_k,\mu_k}(\omega_{\widehat k})
 :=
 \sum_{\substack{(i_j,\mu_j)\in E_j\times[M]\\ j\ne k}}
 a_{(i_1,\mu_1),\ldots,(i_d,\mu_d)}
 \prod_{j\ne k}g_{B_{\mu_j}}^{(i_j)}(\omega_j).
\end{align*}
Then
\[
 M_k(a)
 \le
 a_1^{2(d-1)}
 \int
 \sum_{(i_k,\mu_k)\in E_k\times[M]}
 |c_{i_k,\mu_k}(\omega_{\widehat k})|
 \,d\omega_{\widehat k}.
\]
For every fixed $\omega_{\widehat k}$, Lemma~\ref{lem:degree-one-section3}
inside the block $E_k$ gives
\begin{align*}
 \sum_{(i_k,\mu_k)\in E_k\times[M]}
 |c_{i_k,\mu_k}(\omega_{\widehat k})|
 &\le
 \sqrt M\sup_{\omega_k\in\Omega_{E_k}}
 \Bigg|
 \sum_{(i_k,\mu_k)\in E_k\times[M]}
 c_{i_k,\mu_k}(\omega_{\widehat k})
 g_{B_{\mu_k}}^{(i_k)}(\omega_k)
 \Bigg|\\
 &\le \sqrt M\,\|F\|_\infty.
\end{align*}
Hence
\[
 M_k(a)\le \sqrt M\,a_1^{2(d-1)}\|F\|_\infty
 \qquad(k=1,\ldots,d).
\]
Insert these $d$ estimates into Lemma~\ref{lem:traditional-blei-section3}.
This gives \eqref{eq:scalar-BH-section3}.
\end{proof}

\subsection{One-block operator estimate}

The next elementary estimate is the only point where we compare a one-local
operator with the sum of the norms of its single-site pieces.

\begin{lemma}[A one-local operator estimate]
\label{lem:one-local-op-section3}
Let $E$ be finite and let $H_i\in M_K^0$ for $i\in E$.  Then
\begin{equation}\label{eq:one-local-op-section3}
 \sum_{i\in E}\|H_i\|_{\rm op}
 \le
 4\left\|\sum_{i\in E}H_i^{(i)}\right\|_{\rm op}.
\end{equation}
\end{lemma}

\begin{proof}
We first prove the estimate for traceless Hermitian matrices.  Let
$A_i=A_i^*\in M_K^0$ and put
\[
 a_i=\lambda_{\max}(A_i)\ge0,
 \qquad
 b_i=-\lambda_{\min}(A_i)\ge0.
\]
Since $A_i$ is traceless, its spectrum contains both signs unless $A_i=0$,
and
\[
 \|A_i\|_{\rm op}=\max\{a_i,b_i\}\le a_i+b_i.
\]
The operators $A_i^{(i)}$ act on different tensor factors.  Therefore the
largest and smallest eigenvalues of their sum are the sums of the largest and
smallest eigenvalues:
\[
 \lambda_{\max}\!\left(\sum_iA_i^{(i)}\right)=\sum_i a_i,
 \qquad
 -\lambda_{\min}\!\left(\sum_iA_i^{(i)}\right)=\sum_i b_i.
\]
Consequently,
\begin{align*}
 \left\|\sum_iA_i^{(i)}\right\|_{\rm op}
 =\max\left\{\sum_i a_i,\sum_i b_i\right\}\ge\frac12\sum_i(a_i+b_i)
 \ge\frac12\sum_i\|A_i\|_{\rm op}.
\end{align*}
Thus
\begin{equation}\label{eq:hermitian-one-local-section3}
 \sum_i\|A_i\|_{\rm op}
 \le2\left\|\sum_iA_i^{(i)}\right\|_{\rm op}
\end{equation}
for every traceless Hermitian family.
For general $H_i\in M_K^0$, write
\[
 H_i=A_i+\mathrm i C_i,
 \qquad
 A_i=\frac{H_i+H_i^*}{2},
 \qquad
 C_i=\frac{H_i-H_i^*}{2\mathrm i}.
\]
Both $A_i$ and $C_i$ are traceless Hermitian.  If
$X=\sum_iH_i^{(i)}$, then
\[
 \sum_iA_i^{(i)}=\operatorname{Re}X,
 \qquad
 \sum_iC_i^{(i)}=\operatorname{Im}X.
\]
Using \eqref{eq:hermitian-one-local-section3} twice,
\begin{align*}
 \sum_i\|H_i\|_{\rm op}
 \le\sum_i\|A_i\|_{\rm op}+\sum_i\|C_i\|_{\rm op}\le2\|\operatorname{Re}X\|_{\rm op}
     +2\|\operatorname{Im}X\|_{\rm op}\le4\|X\|_{\rm op},
\end{align*}
which is \eqref{eq:one-local-op-section3}.
\end{proof}

For $E\subset[n]$ define the one-local traceless space
\[
 X_E=\left\{\sum_{i\in E}H_i^{(i)}:H_i\in M_K^0\right\}
\]
and the scalarization map
\[
 \mathcal S_E:X_E\longrightarrow C(\Omega_E),
 \qquad
 \mathcal S_E\!\left(\sum_{i\in E}H_i^{(i)}\right)
 =\sum_{i\in E}g_{H_i}^{(i)}.
\]

\begin{proposition}[Norm of the one-block scalarization]
\label{prop:SE-section3}
For every finite $E$,
\begin{equation}\label{eq:SE-norm-section3}
\|\mathcal S_E\|\le4\sqrt K.
\end{equation}
Moreover this order in $K$ cannot be improved: already for a single site,
\begin{equation}\label{eq:SE-lower-section3}
 \|\mathcal S_{\{i\}}\|\ge\sqrt K.
\end{equation}
Thus the norm of the scalarization is of exact order $\sqrt K$, up to a
universal constant.
\end{proposition}

\begin{proof}
Let $x=\sum_{i\in E}H_i^{(i)}$.  Since $u_{z_i}$ and $u_{w_i}$ are unit
vectors,
\[
 |g_{H_i}^{(i)}(\omega)|
 =\sqrt K\,|\langle H_iu_{z_i},u_{w_i}\rangle|
 \le\sqrt K\,\|H_i\|_{\rm op}.
\]
Hence Lemma~\ref{lem:one-local-op-section3} gives
\[
 \|\mathcal S_Ex\|_\infty
 \le\sqrt K\sum_{i\in E}\|H_i\|_{\rm op}
 \le4\sqrt K\,\|x\|_{\rm op}.
\]
For the lower bound, take one site and let
\[
 U=\operatorname{diag}(1,\zeta,\zeta^2,\ldots,\zeta^{K-1}),
 \qquad
 \zeta=e^{2\pi\mathrm i/K}.
\]
Then $\tr U=0$ and $\|U\|_{\rm op}=1$.  Choose
$z=(1,\ldots,1)$ and $w=(1,\zeta,\ldots,\zeta^{K-1})$.  Then
$Uu_z=u_w$, and therefore
\[
 |g_U(z,w)|=\sqrt K.
\]
This proves \eqref{eq:SE-lower-section3} and shows that the power $K^{1/2}$
in \eqref{eq:SE-norm-section3} is optimal.
\end{proof}

\subsection{Block-transversal estimate}

Let
\[
 [n]=E_1\dot\cup\cdots\dot\cup E_d
\]
be a labelled partition.  A $d$-element set $S\subset[n]$ is called
\emph{block-transversal} if
\[
 |S\cap E_j|=1
 \qquad(j=1,\ldots,d).
\]
An exact-support operator is block-transversal if every support occurring in
its tensor-product basis expansion is block-transversal.

\begin{proposition}[Block-transversal estimate]
\label{prop:block-transversal-section3}
Let $A\in\mathcal H_{d,K}(n)$ be block-transversal with respect to
$E_1,\ldots,E_d$.  Then $A$ has a unique expansion
\begin{equation}\label{eq:block-transversal-explicit-section3}
\begin{aligned}
 A
 ={}&\sum_{(i_1,\mu_1)\in E_1\times[M]}\cdots
 \sum_{(i_d,\mu_d)\in E_d\times[M]}
 a_{(i_1,\mu_1),\ldots,(i_d,\mu_d)}
 \prod_{j=1}^d B_{\mu_j}^{(i_j)},
\end{aligned}
\end{equation}
where the product means that $B_{\mu_j}$ acts at the site $i_j$ and the
identity acts at every other site.  Moreover,
\[
 \|\widehat A_{\mathcal B}\|_{\ell_{p_d}}
 \le
 \sqrt M\,a_1^{2(d-1)}(4\sqrt K)^d\|A\|_{\rm op}.
 \]
\end{proposition}

\begin{proof}
For $j=1,\ldots,d$, let
$\omega_j=((z_i,w_i))_{i\in E_j}\in\Omega_{E_j}$.  Scalarize
\eqref{eq:block-transversal-explicit-section3} block by block and define
\begin{align*}
 F_A(\omega_1,\ldots,\omega_d)
 :={}&\sum_{(i_1,\mu_1)\in E_1\times[M]}\cdots
 \sum_{(i_d,\mu_d)\in E_d\times[M]}
 a_{(i_1,\mu_1),\ldots,(i_d,\mu_d)}
 \prod_{j=1}^d g_{B_{\mu_j}}^{(i_j)}(\omega_j).
\end{align*}
Fix $\omega_1,\ldots,\omega_d$.  The map
\[
 \varphi_j:X_{E_j}\to\C,
 \qquad
 \varphi_j(x)=\mathcal S_{E_j}x(\omega_j),
\]
has norm at most $4\sqrt K$ by Proposition~\ref{prop:SE-section3}.  Hence,
by the definition of the injective tensor norm,
\[
 |F_A(\omega_1,\ldots,\omega_d)|
 \le(4\sqrt K)^d\|A\|_\varepsilon.
\]
The injective Banach-space tensor norm is dominated by the minimal
$C^*$-tensor norm.  Since the blocks $E_1,\ldots,E_d$ act on disjoint tensor
factors, the latter is exactly the concrete operator norm of $A$.  Thus
\begin{equation}\label{eq:FA-sup-section3}
 \|F_A\|_\infty\le(4\sqrt K)^d\|A\|_{\rm op}.
\end{equation}
Applying Proposition~\ref{prop:scalar-BH-section3} to the explicit coefficient
array in \eqref{eq:block-transversal-explicit-section3}, and then using
\eqref{eq:FA-sup-section3}, gives
\begin{align*}
 \|\widehat A_{\mathcal B}\|_{\ell_{p_d}}
 &=
 \left(
 \sum_{(i_1,\mu_1)\in E_1\times[M]}\cdots
 \sum_{(i_d,\mu_d)\in E_d\times[M]}
 |a_{(i_1,\mu_1),\ldots,(i_d,\mu_d)}|^{p_d}
 \right)^{1/p_d}\\[2ex]
 &\le \sqrt M\,a_1^{2(d-1)}\|F_A\|_\infty\le \sqrt M\,a_1^{2(d-1)}(4\sqrt K)^d\|A\|_{\rm op},
\end{align*}
as claimed.
\end{proof}

\subsection{Block-transversal decoupling}
\label{Block-transversal decoupling}

We now adapt the support-decoupling device from
\cite{DGMMM-support} to the operator setting.  The point is to recover an
arbitrary exact-support coefficient vector by averaging block-transversal
pieces.

\begin{lemma}[Block-transversal decoupling]
\label{lem:block-transversal-decoupling}
Let \(A\in\mathcal H_{d,K}(n)\).  Choose a random labelling
\[
\ell:[n]\longrightarrow[d]
\]
by assigning independently to every site a uniform label in
\(\{1,\ldots,d\}\), and put
\[
E_j=\ell^{-1}(j),
\qquad
\Delta_\ell
=
\prod_{j=1}^d
\bigl(I-\mathbb E_{E_j}\bigr),
\qquad
A_\ell=\Delta_\ell A,
\]
where \(\mathbb E_E\) applies the normalized trace at the sites in \(E\)
and leaves all other tensor factors unchanged.  Then \(A_\ell\) is
block-transversal with respect to \(E_1,\ldots,E_d\),
\begin{equation}
\label{eq:decoupled-operator-norm}
\|A_\ell\|_{\rm op}
\le
2^d\|A\|_{\rm op},
\end{equation}
and
\begin{equation}
\label{eq:decoupled-coefficient-recovery}
\mathbb E_\ell
\|\widehat{A_\ell}_{\mathcal B}\|_{\ell_{p_d}}^{p_d}
=
\frac{d!}{d^d}
\|\widehat A_{\mathcal B}\|_{\ell_{p_d}}^{p_d}.
\end{equation}
\end{lemma}

\begin{proof}
Since \(\tau_K(B_\mu)=0\) for \(\mu\ge1\), every tensor-product basis
element satisfies
\[
\mathbb E_E(B_\beta)
=
\begin{cases}
B_\beta,
&
\supp\beta\cap E=\varnothing,
\\
0,
&
\supp\beta\cap E\ne\varnothing.
\end{cases}
\]
Consequently,
\[
\Delta_\ell(B_\beta)
=
\begin{cases}
B_\beta,
&
\supp\beta\cap E_j\ne\varnothing
\quad\text{for every }j=1,\ldots,d,
\\
0,
&
\text{otherwise}.
\end{cases}
\]
Since \(A\) has exact support \(d\), a support occurring in its expansion
meets every \(E_j\) if and only if it meets each of them in exactly one
site.  Thus \(A_\ell\) is block-transversal.
Moreover, the conditional expectations are contractions, and hence
\[
\|\Delta_\ell\|
\le
\prod_{j=1}^d
\|I-\mathbb E_{E_j}\|
\le
2^d,
\]
which proves \eqref{eq:decoupled-operator-norm}.
Finally, coefficientwise,
\[
\widehat{A_\ell}_{\mathcal B}(\beta)
=
\mathbf 1_{\{
\ell|_{\supp\beta}\text{ is bijective}
\}}
\widehat A_{\mathcal B}(\beta).
\]
For a fixed \(d\)-element support \(S\subset[n]\), the restriction
\(\ell|_S\) is bijective with probability
\(
\frac{d!}{d^d}.
\)
Therefore, by Tonelli's theorem,
\[
\begin{aligned}
\mathbb E_\ell
\|\widehat{A_\ell}_{\mathcal B}\|_{\ell_{p_d}}^{p_d}
=
\sum_{|\supp\beta|=d}
|\widehat A_{\mathcal B}(\beta)|^{p_d}
\,
\mathbb P\{
\ell|_{\supp\beta}\text{ is bijective}
\}
=
\frac{d!}{d^d}
\|\widehat A_{\mathcal B}\|_{\ell_{p_d}}^{p_d},
\end{aligned}
\]
as claimed.
\end{proof}

We can now assemble the preceding estimates.

\begin{proof}
\label{Proof of the main result}
    Put \(M=K^2-1\).  For every labelling \(\ell\),
Lemma~\ref{lem:block-transversal-decoupling} and
Proposition~\ref{prop:block-transversal-section3} give
\[
\begin{aligned}
\|\widehat{A_\ell}_{\mathcal B}\|_{\ell_{p_d}}
\le
\sqrt M\,a_1^{2(d-1)}(4\sqrt K)^d
\|A_\ell\|_{\rm op}
\le
\sqrt M\,a_1^{2(d-1)}(8\sqrt K)^d
\|A\|_{\rm op}.
\end{aligned}
\]
Using \eqref{eq:decoupled-coefficient-recovery}, we therefore obtain
\[
\begin{aligned}
\frac{d!}{d^d}
\|\widehat A_{\mathcal B}\|_{\ell_{p_d}}^{p_d}
=
\mathbb E_\ell
\|\widehat{A_\ell}_{\mathcal B}\|_{\ell_{p_d}}^{p_d}
\le
\left(
\sqrt M\,a_1^{2(d-1)}(8\sqrt K)^d
\|A\|_{\rm op}
\right)^{p_d}.
\end{aligned}
\]
Taking \(p_d\)-th roots and recalling \(M=K^2-1\) gives
\[
\|\widehat A_{\mathcal B}\|_{\ell_{p_d}}
\le
\sqrt{K^2-1}\,
a_1^{2(d-1)}
(8\sqrt K)^d
\left(\frac{d^d}{d!}\right)^{1/p_d}
\|A\|_{\rm op},
\]
which is the asserted estimate.
Finally, taking \(d\)-th roots and using
\[
(K^2-1)^{1/(2d)}\longrightarrow1,
\qquad
a_1^{2(d-1)/d}\longrightarrow a_1^2,
\qquad
\left(\frac{d^d}{d!}\right)^{1/(dp_d)}
\longrightarrow\sqrt e,
\]
we obtain
\(
\beta(K)
\le
8a_1^2\sqrt e\,\sqrt K.
\)
\end{proof}


\subsection{From exact support to the full $d$-local space}
\label{subsec:full-d-local-BH}

The exact-support estimate also yields a Bohnenblust--Hille inequality on the
whole $d$-local space
\[
\mathcal H_{\le d,K}(n)
=
\bigoplus_{r=0}^d\mathcal H_{r,K}(n).
\]
The only additional point is to control the norm of each exact-support
component.
For $r\ge1$, let $C_{r,K}$ denote the exact-support constant defined above, and
put $C_{0,K}=1$.

\begin{theorem}[Bohnenblust--Hille inequality for $d$-local operators]
\label{thm:local-BH-leq-d}
Let $K\ge2$ and $d\ge2$.  Then, uniformly in $n$ and in the normalized
Hilbert--Schmidt orthonormal local operator basis $\mathcal B$, every
\(
A\in\mathcal H_{\le d,K}(n)
\)
satisfies
\begin{equation}\label{eq:local-BH-leq-d}
 \|\widehat A_{\mathcal B}\|_{\ell_{p_d}}
 \le
 8^d
 \left(\sum_{r=0}^d C_{r,K}\right)
 \|A\|_{\rm op},
 \qquad
 p_d=\frac{2d}{d+1}.
\end{equation}
Consequently, if $C_{\leq d,K}$ denotes the least constant in
\eqref{eq:local-BH-leq-d}, then
\[
 \limsup_{d\to\infty}C_{\leq d,K}^{1/d}
 \le
 8\max\{1,\beta(K)\}.
\]
In particular, by Theorem~\ref{thm:local-BH},
\begin{equation}\label{eq:beta-leq-d-sqrtK}
\limsup_{d\to\infty}C_{\leq d,K}^{1/d}
 \le
 64a_1^2\sqrt e\,\sqrt K.
\end{equation}
Since \(H_{d,K}(n)\subset H_{\le d,K}(n)\), we also have
\(C_{d,K}\le C_{\le d,K}\). Together with Theorem~\ref{thm:beta-lower} below, this gives
\[
 cK^{1/4}
 \le
 \limsup_{d\to\infty} C_{\le d,K}^{1/d}
 \le
 64a_1^2\sqrt e\,\sqrt K.
\]
Thus the exponential base for the full \(d\)-local space is polynomially
constrained between orders \(K^{1/4}\) and \(K^{1/2}\).
\end{theorem}

\begin{proof}
Write the support decomposition of $A$ as
\[
 A=A_0+A_1+\cdots+A_d,
 \qquad
 A_r\in\mathcal H_{r,K}(n).
\]
We first show that
\[
 \|A_r\|_{\rm op}\le8^d\|A\|_{\rm op},
 \qquad 0\le r\le d.
\]
For $0\le t\le1$, define the one-site map
\[
 \Phi_t:M_K(\mathbb C)\longrightarrow M_K(\mathbb C),
 \qquad
 \Phi_t(X)
 =
 tX+(1-t)\tau_K(X)I.
\]
This is a convex combination of the identity and the normalized-trace
conditional expectation, and hence
\[
 \|\Phi_t\|=1.
\]
Moreover,
\[
 \Phi_t(I)=I,
 \qquad
 \Phi_t(H)=tH
 \quad(H\in M_K^0).
\]
It follows that
\[
 \Phi_t^{\otimes n}(A)
 =
 \sum_{r=0}^d t^rA_r.
\]
Since $\Phi_t^{\otimes n}$ is a contraction,
\[
 \left\|\sum_{r=0}^d t^rA_r\right\|_{\rm op}
 \le
 \|A\|_{\rm op},
 \qquad 0\le t\le1.
\]
We shall use the following standard coefficient estimate for polynomials on an interval (see, e.g., \cite[Chapter~2]{RivlinChebyshev}). If \[ P(t)=\sum_{r=0}^d x_r t^r \] is a polynomial with values in a Banach space $X$, then there exists a universal constant, which may be taken to be \(8\),
such that
\[
\|x_r\|\le 8^d\sup_{0\le t\le1}\|P(t)\|,
\qquad 0\le r\le d.
\]
Indeed, the Banach-space-valued version follows immediately from the scalar one by the Hahn--Banach theorem.
We now apply the exact-support Bohnenblust--Hille inequality to every layer.
For $1\le r\le d$,
\[
 p_r=\frac{2r}{r+1}\le\frac{2d}{d+1}=p_d,
\]
and hence
\[
 \|\widehat {A_r}_{\mathcal B}\|_{\ell_{p_d}}
 \le
 \|\widehat {A_r}_{\mathcal B}\|_{\ell_{p_r}}
 \le
 C_{r,K}\|A_r\|_{\rm op}.
\]
For $r=0$ the same estimate holds with $C_{0,K}=1$.  Since the coefficient
supports of the different $A_r$ are disjoint,
\begin{align*}
 \|\widehat A_{\mathcal B}\|_{\ell_{p_d}}
 \le
 \sum_{r=0}^d
 \|\widehat {A_r}_{\mathcal B}\|_{\ell_{p_d}}\le
 \sum_{r=0}^d C_{r,K}\|A_r\|_{\rm op}
&\le
 8^d
 \left(\sum_{r=0}^dC_{r,K}\right)
 \|A\|_{\rm op}.
\end{align*}
This proves \eqref{eq:local-BH-leq-d}.
Finally,
\[
 C_{\leq d,K}
 \le
 8^d\sum_{r=0}^dC_{r,K}.
\]
Taking $d$th roots and using
\[
 \beta(K)=\limsup_{r\to\infty}C_{r,K}^{1/r},
\]
we obtain
\[
 \limsup_{d\to\infty}C_{\leq d,K}^{1/d}
 \le
 8\max\{1,\beta(K)\}.
\]
The estimate \eqref{eq:beta-leq-d-sqrtK} now follows immediately from
Theorem~\ref{thm:local-BH}.
\end{proof}

\subsection{A Slote-type lower bound for the exponential base}
\label{subsec:lower-bound}

We now complement the upper bound by a polynomial lower bound in the local
dimension. Recall that
\[
 \beta(K)=\limsup_{d\to\infty} C_{d,K}^{1/d},
\]
where \(C_{d,K}\) is the exact-support Bohnenblust--Hille constant, uniform
in the number of sites and in the choice of the normalized
Hilbert--Schmidt orthonormal local operator basis. Our goal in this
subsection is the following.

\begin{theorem}[Lower bound for the exponential base]
\label{thm:beta-lower}
There is a universal constant \(c>0\) such that, for every \(K\ge2\),
\[
 \beta(K)\ge cK^{1/4}.
\]
Equivalently,
\[
 \beta(K)=\Omega(K^{1/4}).
\]
\end{theorem}

The proof has two ingredients. First, a quantitative form of Dvoretzky's
theorem provides \(q\simeq K\) traceless Hermitian directions on which the
operator norm is uniformly dominated by the Euclidean norm. Second, a
recursive Hadamard amplification turns such a \(q\)-dimensional Euclidean
family into \(q^{d(d+1)/2}\) flat coefficients at exact support \(d\).
The latter step is elementary and is the source of the exponent \(1/4\).
We begin with the geometric input.

\begin{lemma}[Euclidean core]
\label{lem:euclidean-core}
There are universal constants \(c_0>0\) and \(D\ge1\) such that, for every
\(K\ge2\), one can find a power of two \(q\ge c_0K\) and traceless
Hermitian matrices
\[
 A_1,\ldots,A_q\in M_K^0
\]
which are orthonormal for the normalized Hilbert--Schmidt inner product and
satisfy
\begin{equation}\label{eq:euclidean-core}
 \left\|\sum_{a=1}^q t_aA_a\right\|_{\rm op}
 \le
 D\left(\sum_{a=1}^q |t_a|^2\right)^{1/2}
 \qquad (t_1,\ldots,t_q\in\mathbb R).
\end{equation}
\end{lemma}

\begin{proof}
Consider the real vector space
\[
 \mathfrak h_K^0
 =
 \{A\in M_K(\mathbb C):A=A^*,\ \tr A=0\},
\]
equipped with the operator norm. We use the normalized Hilbert--Schmidt
norm
\[
 \|A\|_2
 =
 \left(\frac1K\tr(A^2)\right)^{1/2}
\]
as its Euclidean norm. Its real dimension is
\(
 N=K^2-1.
\)
Let
\[
 S_2=\{A\in\mathfrak h_K^0:\|A\|_2=1\},
\]
let \(\sigma\) denote normalized surface measure on \(S_2\), and put
\[
 M_K=\int_{S_2}\|A\|_{\rm op}\,d\sigma(A),
 \qquad
 b_K=\sup_{A\in S_2}\|A\|_{\rm op}.
\]
Since the largest singular value dominates the root mean square of the
singular values,
\[
 1=\|A\|_2\le\|A\|_{\rm op}\le\sqrt K\,\|A\|_2=\sqrt K
 \qquad (A\in S_2).
\]
Thus
\[
 M_K\ge1,
 \qquad
 b_K\le\sqrt K.
\]
We also need a universal upper bound for \(M_K\). Let \(G\) be a standard
Gaussian vector in the Euclidean space
\((\mathfrak h_K^0,\|\cdot\|_2)\). Writing
\[
 G=R\Theta,
\]
where \(\Theta\) is uniformly distributed on \(S_2\), \(R=\|G\|_2\), and
\(R\) and \(\Theta\) are independent, gives
\[
 \mathbb E\|G\|_{\rm op}
 =
 (\mathbb ER)M_K.
\]
Since \(G\) lives in a Euclidean space of dimension \(N=K^2-1\),
\[
 \mathbb ER\asymp\sqrt N\asymp K.
\]
On the other hand, the standard Gaussian operator-norm estimate for
Hermitian random matrices gives
\[
 \mathbb E\|G\|_{\rm op}\lesssim K;
\]
see, for example, \cite[Chapter~4]{Vershynin2018}. Hence
\[
 M_K\le C
\]
with a universal constant \(C\).
We now apply the quantitative form of Dvoretzky's theorem in Milman's
mean-norm formulation; see \cite{Milman1971}. For a fixed numerical
distortion it gives a subspace \(E\subset\mathfrak h_K^0\) of dimension
\[
 \dim E
 \ge
 c\left(\frac{M_K}{b_K}\right)^2N
 \ge cK
\]
on which
\[
 \|A\|_{\rm op}\le D\|A\|_2
 \qquad (A\in E)
\]
with a universal constant \(D\).
Choose a normalized Hilbert--Schmidt orthonormal family of cardinality
\(q_0\ge cK\) in \(E\), and let \(q\) be the largest power of two with
\(q\le q_0\). Then \(q\ge q_0/2\), after adjusting the universal constant,
and the first \(q\) matrices satisfy~\eqref{eq:euclidean-core}.
\end{proof}

We next isolate the recursive amplification mechanism. Extend
\[
 I,A_1,\ldots,A_q
\]
to a normalized Hilbert--Schmidt orthonormal basis
\[
 \mathcal B=\{I,B_1,\ldots,B_{K^2-1}\}
\]
consisting of Hermitian matrices. This is possible because the Hermitian
matrices form a real Hilbert space of dimension \(K^2\); a real
orthonormal basis of this space is also a complex orthonormal basis of
\(M_K(\mathbb C)\).

\begin{lemma}[Recursive flat family]
\label{lem:recursive-flat-family}
Assume that \(A_1,\ldots,A_q\) satisfy~\eqref{eq:euclidean-core} with
\(q\) a power of two.
For every integer \(r\ge1\) there are an integer \(n_r\) and Hermitian
operators
\[
 P_1^{(r)},\ldots,P_{L_r}^{(r)}
 \in\mathcal H_{r,K}(n_r),
\]
where
\[
L_r:=q^r,\qquad
N_r:=q^{r(r-1)/2},
\qquad
n_r:=\frac{q^r-1}{q-1},
\]
with the following properties:
\begin{enumerate}
\item[(a)]
the supports of the coefficient vectors
\(\widehat{P_j^{(r)}}_{\mathcal B}\) are pairwise disjoint;

\item[(b)]
each \(P_j^{(r)}\) has exactly \(N_r\) nonzero
\(\mathcal B^{\otimes n_r}\)-coefficients, all of modulus
\(N_r^{-1/2}\);

\item[(c)]
for every real family \(t=(t_j)_{j=1}^{L_r}\),
\begin{equation}\label{eq:recursive-Euclidean-estimate}
 \Bigg\|
 \sum_{j=1}^{L_r}t_jP_j^{(r)}
 \Bigg\|_{\rm op}
 \le
 D^r
 \Bigg(\sum_{j=1}^{L_r}|t_j|^2\Bigg)^{1/2}.
\end{equation}
\end{enumerate}
\end{lemma}

\begin{proof}
We proceed by induction on \(r\).
For \(r=1\), take \(n_1=1\) and
\[
 P_a^{(1)}=A_a,
 \qquad 1\le a\le q.
\]
Then \(L_1=q\), \(N_1=1\), the coefficient supports are disjoint, and
all assertions follow from Lemma~\ref{lem:euclidean-core}.
Suppose now that the family has been constructed at level \(r\). Write
\[
 L=L_r,
 \qquad
 P_j=P_j^{(r)}
 \quad (1\le j\le L).
\]
Since
\(
 L=q^r
\)
is a power of two, there exists a Sylvester Hadamard matrix
\[
 H=(h_{ij})_{i,j=1}^L,
 \qquad
 h_{ij}\in\{-1,1\},
 \qquad
 HH^{\mathsf T}=LI.
\]
Define
\begin{equation}\label{eq:Hadamard-mixtures}
 R_i
 =
 \frac1{\sqrt L}\sum_{j=1}^L h_{ij}P_j,
 \qquad
 1\le i\le L.
\end{equation}
The Hadamard mixing preserves the Euclidean operator-norm estimate.
Indeed, for \(s=(s_i)_{i=1}^L\in\mathbb R^L\),
\[
 \sum_{i=1}^Ls_iR_i
 =
 \sum_{j=1}^L u_jP_j,
 \qquad
 u_j
 =
 \frac1{\sqrt L}\sum_{i=1}^Ls_ih_{ij}.
\]
Since \(H/\sqrt L\) is orthogonal,
\[
 \sum_{j=1}^L|u_j|^2
 =
 \sum_{i=1}^L|s_i|^2.
\]
Thus, by the induction hypothesis,
\begin{equation}\label{eq:Hadamard-Euclidean-estimate}
 \left\|\sum_{i=1}^Ls_iR_i\right\|_{\rm op}
 \le
 D^r
 \left(\sum_{i=1}^L|s_i|^2\right)^{1/2}.
\end{equation}
We now append \(L\) fresh tensor sites, labelled \(1,\ldots,L\).
Since \(L=L_r=q^r\), we have
\[
 n_{r+1}=n_r+L.
\]
For \(1\le i\le L\) and \(1\le a\le q\), let \(A_a^{[i]}\) denote the
operator on the fresh block which acts as \(A_a\) on the \(i\)-th fresh
site and as the identity on all other fresh sites. Define
\[
 P_{i,a}^{(r+1)}
 =
 R_i\otimes A_a^{[i]}.
\]
There are
\(
 qL=q^{r+1}
\)
such operators, so \(L_{r+1}=q^{r+1}\).
We first verify the operator-norm estimate. Let
\(t_{i,a}\in\mathbb R\), and define on the fresh block
\[
 X_i
 =
 \sum_{a=1}^q t_{i,a}A_a^{[i]}.
\]
Then
\begin{equation}\label{eq:recursive-tensor-sum}
 \sum_{i=1}^L\sum_{a=1}^q
 t_{i,a}P_{i,a}^{(r+1)}
 =
 \sum_{i=1}^L R_i\otimes X_i.
\end{equation}
The operators \(X_1,\ldots,X_L\) are Hermitian and act on distinct
tensor factors. Hence they commute and admit a simultaneous spectral
decomposition. On a common spectral subspace write
\[
 X_i=\lambda_iI,
 \qquad
 \lambda_i\in\mathbb R.
\]
Restricting~\eqref{eq:recursive-tensor-sum} to that subspace leaves the
operator
\[
 \sum_{i=1}^L\lambda_iR_i
\]
on the old tensor factors. By~\eqref{eq:Hadamard-Euclidean-estimate},
\[
 \left\|\sum_{i=1}^L\lambda_iR_i\right\|_{\rm op}
 \le
 D^r
 \left(\sum_{i=1}^L|\lambda_i|^2\right)^{1/2}.
\]
Moreover, Lemma~\ref{lem:euclidean-core} gives
\[
 |\lambda_i|
 \le
 \|X_i\|_{\rm op}
 \le
 D
 \left(\sum_{a=1}^q|t_{i,a}|^2\right)^{1/2}.
\]
Taking the supremum over all common spectral subspaces therefore yields
\[
 \left\|
 \sum_{i=1}^L\sum_{a=1}^q
 t_{i,a}P_{i,a}^{(r+1)}
 \right\|_{\rm op}
 \le
 D^{r+1}
 \left(
 \sum_{i=1}^L\sum_{a=1}^q|t_{i,a}|^2
 \right)^{1/2}.
\]
This is~\eqref{eq:recursive-Euclidean-estimate} at level \(r+1\).

It remains to check the support and coefficient assertions. By induction,
the coefficient supports of \(P_1,\ldots,P_L\) are pairwise disjoint,
and each \(P_j\) has \(N_r\) coefficients of modulus \(N_r^{-1/2}\).
Consequently, no cancellation is possible in~\eqref{eq:Hadamard-mixtures}:
every \(R_i\) has exactly
\[
 LN_r
\]
nonzero coefficients, all of modulus
\[
 \frac1{\sqrt L}\frac1{\sqrt{N_r}}
 =
 (LN_r)^{-1/2}.
\]
Multiplication by \(A_a^{[i]}\) adds precisely one nonidentity factor on
a fresh site. Hence every \(P_{i,a}^{(r+1)}\) has exact support \(r+1\).

Finally, the coefficient supports of two different
\(P_{i,a}^{(r+1)}\) are disjoint. If the indices \(i\) are different,
their new active sites are different; if the indices \(i\) agree but
the indices \(a\) are different, the local basis element on that fresh
site is different. Thus
\[
 N_{r+1}=L_rN_r.
\]
Starting from \(L_1=q\) and \(N_1=1\), the recursions
\[
 L_{r+1}=qL_r,
 \qquad
 N_{r+1}=L_rN_r
\]
give
\[
 L_r=q^r,
 \qquad
 N_r=q^{1+2+\cdots+(r-1)}
     =q^{r(r-1)/2}.
\]
This proves the lemma.
\end{proof}

\smallskip

\begin{proof}[Proof of Theorem~\ref{thm:beta-lower}]
Fix \(K\), and choose \(q\), \(D\), and the local basis \(\mathcal B\)
as above. For \(d\ge1\), apply
Lemma~\ref{lem:recursive-flat-family} at level \(d\), and define
\[
 V_d
 =
 L_d^{-1/2}
 \sum_{j=1}^{L_d}P_j^{(d)}
 \in\mathcal H_{d,K}(n_d).
\]
Using~\eqref{eq:recursive-Euclidean-estimate} with the constant coefficient
vector \(t_j=L_d^{-1/2}\), we obtain
\begin{equation}\label{eq:Vd-operator-bound}
 \|V_d\|_{\rm op}\le D^d.
\end{equation}
On the other hand, the coefficient supports of the
\(P_j^{(d)}\)'s are pairwise disjoint. Hence \(V_d\) has exactly
\begin{equation}\label{eq:Md-number-coefficients}
 M_d
 =
 L_dN_d
 =
 q^d q^{d(d-1)/2}
 =
 q^{d(d+1)/2}
\end{equation}
nonzero coefficients. Every one of them has modulus
\[
 L_d^{-1/2}N_d^{-1/2}
 =
 M_d^{-1/2}.
\]
Therefore, with \(p_d=2d/(d+1)\),
\[
 \|\widehat V_{d,\mathcal B}\|_{\ell_{p_d}}
 =
 \left(
 M_d(M_d^{-1/2})^{p_d}
 \right)^{1/p_d}
 =
 M_d^{1/p_d-1/2}.
\]
Since
\[
 \frac1{p_d}-\frac12=\frac1{2d},
\]
equation~\eqref{eq:Md-number-coefficients} gives
\begin{equation}\label{eq:Vd-lpd-final}
 \|\widehat V_{d,\mathcal B}\|_{\ell_{p_d}}
 =
 M_d^{1/(2d)}
 =
 q^{(d+1)/4}.
\end{equation}
The constant \(C_{d,K}\) is required to work uniformly over all
normalized Hilbert--Schmidt orthonormal local bases. In particular, it
must dominate the ratio corresponding to the specific basis
\(\mathcal B\) constructed above. Combining
\eqref{eq:Vd-operator-bound} and \eqref{eq:Vd-lpd-final}, we obtain
\[
 C_{d,K}
 \ge
 D^{-d}q^{(d+1)/4}.
\]
Taking \(d\)-th roots yields
\[
 C_{d,K}^{1/d}
 \ge
 D^{-1}q^{(d+1)/(4d)}.
\]
Hence
\[
 \beta(K)
 =
 \limsup_{d\to\infty}C_{d,K}^{1/d}
 \ge
 D^{-1}q^{1/4}.
\]
By Lemma~\ref{lem:euclidean-core}, \(q\ge c_0K\), after adjusting the
universal constant. Therefore
\[
 \beta(K)\ge cK^{1/4},
\]
as claimed.
\end{proof}

\begin{remark}[Where the exponent \(1/4\) comes from]
\label{rem:quarter-exponent}
The numerical mechanism is particularly transparent. At exact support
\(d\), the recursion produces
\[
 M_d=q^{d(d+1)/2}
\]
flat coefficients. A flat \(\ell_2\)-normalized vector with \(M_d\)
coordinates has, at the Bohnenblust--Hille exponent,
\[
 \ell_{p_d}\text{-norm}
 =
 M_d^{1/(2d)}.
\]
The definition of \(\beta(K)\) requires one further \(d\)-th root.
Thus the exponent of \(q\) is
\[
 \frac{d(d+1)}2\cdot\frac1{2d}\cdot\frac1d
 =
 \frac{d+1}{4d}
 \longrightarrow
 \frac14.
\]
Since the Euclidean core has \(q\simeq K\) directions, this yields the
power \(K^{1/4}\).
\end{remark}

\begin{remark}
\label{rem:slote-recursion}
The amplification is in the spirit of the flat-spectrum construction of
Slote~\cite{Slote2026}, but the organization is different. Instead of
attaching a separate suffix to a fixed core, the Hadamard mixing is
iterated one support level at a time. Each step preserves the Euclidean
operator-norm estimate up to one additional factor \(D\), while adding
only one active tensor site to every coefficient.
For \(K=2\), one may take \(A_1=X\) and \(A_2=Z\), the usual Pauli
matrices. Then
\[
 \|t_1X+t_2Z\|_{\rm op}
 =
 (|t_1|^2+|t_2|^2)^{1/2}
 \qquad (t_1,t_2\in\mathbb R),
\]
so that \(q=2\) and \(D=1\). The same recursion therefore gives the
explicit bound
\[
 \beta(2)\ge2^{1/4}.
\]
\end{remark} 

The preceding argument leaves a gap between the lower bound and the upper bound in the dependence of the
exponential base on the local dimension. It is therefore natural to ask
for the correct asymptotic behavior of $\beta(K)$ as $K\to\infty$.
In other words, can the gap
$$
cK^{1/4}\le \beta(K)\le C\sqrt K
$$
be closed?


\section{GM and HW scalarization procedures of SVZ revisited}
\label{sec:HW-support-sensitive}

The argument of Section~\ref{sec:BH-theory} is intrinsic: it is formulated directly on the
support spaces \(\mathcal H_{\le d,K}(n)\) and does not rely on a
distinguished local operator basis.  This is in contrast to the
scalarization methods of Slote--Volberg--Zhang~\cite{SVZqudit,
VZnoncomm,
SVZtightness}, which start from the
Gell--Mann or Heisenberg--Weyl coordinates and reduce the operator problem
to a commutative one.

In this section we relate the two viewpoints.  We first show that, at the
Bohnenblust--Hille exponent, changing the local orthonormal operator basis
costs only a factor independent of the interaction order.  We then revisit
the GM and HW scalarizations of Slote--Volberg--Zhang to identify precisely
what these basis-dependent reductions imply for the intrinsic support
spaces, and how they compare with the direct basis-free estimate obtained
above.

\subsection{Change of tensor-product operator coordinates}
\label{subsec:change-of-coordinates}

The Bohnenblust--Hille exponent is stable under an arbitrary change of
single-site orthonormal operator coordinates.  The important point is that,
although the change of coordinates is tensorized over the active sites, at the
exponent
\(
p_d=2d/(d+1)\)
the resulting loss is bounded independently of the number of sites.

\begin{theorem}[Change of coordinates]
\label{thm:change-of-coordinates}
Let
\[
\mathcal B=\{I,B_1,\ldots,B_{K^2-1}\},
\qquad
\mathcal C=\{I,C_1,\ldots,C_{K^2-1}\}
\]
be two orthonormal bases of $M_K(\mathbb C)$ for the normalized
Hilbert--Schmidt inner product, with
\(
B_\mu,C_\mu\in M_K^0
\)
for
\(
1\le \mu\le K^2-1.
\)
Let $\mathcal B^{\otimes n}$ and $\mathcal C^{\otimes n}$ be the corresponding
tensor-product operator bases.  Then, for every
\(
A\in\mathcal H_{\le d,K}(n),
\)
one has
\begin{equation}\label{eq:change-of-coordinates}
 (K^2-1)^{-1/2}
 \|\widehat A_{\mathcal C}\|_{\ell_{p_d}}
 \le
 \|\widehat A_{\mathcal B}\|_{\ell_{p_d}}
 \le
 (K^2-1)^{1/2}
 \|\widehat A_{\mathcal C}\|_{\ell_{p_d}},
 \qquad
 p_d=\frac{2d}{d+1}.
\end{equation}
In particular, the constants are independent of $n$, $d$, and of the two
orthonormal bases.
\end{theorem}

\begin{proof}
Put
\(
m=K^2-1.
\)
Since $\{B_1,\ldots,B_m\}$ and $\{C_1,\ldots,C_m\}$ are orthonormal bases of
the same Hilbert space $M_K^0$, there is a unitary matrix
\[
U=(u_{\mu\nu})_{\mu,\nu=1}^m\in U(m)
\]
such that
\(
C_\nu=\sum_{\mu=1}^m u_{\mu\nu}B_\mu.
\) Fix a support
\(
S=\{j_1,\ldots,j_r\}\subset[n],
\qquad r\le d,
\)
and let $A_S\in\mathcal H_S$ be the corresponding support component of $A$.
In $\mathcal C$-coordinates,
\[
A_S
=
\sum_{\nu_1,\ldots,\nu_r=1}^m
\widehat A_{\mathcal C,S}(\nu_1,\ldots,\nu_r)
C_{\nu_1}^{(j_1)}\otimes\cdots\otimes C_{\nu_r}^{(j_r)}.
\]
Substituting
\(
C_{\nu_\ell}
=
\sum_{\mu_\ell=1}^m
u_{\mu_\ell\nu_\ell}B_{\mu_\ell}
\)
at every active site gives
\[
A_S
=
\sum_{\mu_1,\ldots,\mu_r=1}^m
\left(
\sum_{\nu_1,\ldots,\nu_r=1}^m
u_{\mu_1\nu_1}\cdots u_{\mu_r\nu_r}
\widehat A_{\mathcal C,S}(\nu_1,\ldots,\nu_r)
\right)
B_{\mu_1}^{(j_1)}\otimes\cdots\otimes B_{\mu_r}^{(j_r)}.
\]
Therefore, in tensor notation,
\begin{equation*}
\widehat A_{\mathcal B,S}
=
U^{\otimes r}\widehat A_{\mathcal C,S}.
\end{equation*}
The operator $U^{\otimes r}$ is unitary on $\ell_2^{m^r}$.  Since
$1\le p_d\le2$, for every $x\in\mathbb C^{m^r}$,
\[
\|U^{\otimes r}x\|_{p_d}
\le
m^{r(1/p_d-1/2)}\|U^{\otimes r}x\|_2
=
m^{r(1/p_d-1/2)}\|x\|_2
\le
m^{r(1/p_d-1/2)}\|x\|_{p_d}.
\]
Now
\(
\frac1{p_d}-\frac12
=
\frac1{2d},
\)
and hence, since $r\le d$,
\[
m^{r(1/p_d-1/2)}
=
m^{r/(2d)}
\le
m^{1/2}.
\]
Thus, on every support $S$,
\[
\|\widehat A_{\mathcal B,S}\|_{p_d}
\le
m^{1/2}
\|\widehat A_{\mathcal C,S}\|_{p_d}.
\]
The coefficient sets corresponding to distinct supports are disjoint.
Therefore, summing the $p_d$-powers over all supports with $|S|\le d$ gives
\[
\|\widehat A_{\mathcal B}\|_{p_d}
\le
m^{1/2}\|\widehat A_{\mathcal C}\|_{p_d}.
\]
Interchanging $\mathcal B$ and $\mathcal C$ gives the reverse inequality.
Since $m=K^2-1$, this proves \eqref{eq:change-of-coordinates}.
\end{proof}

\subsection{The GM route}

The change-of-coordinates estimate from
Theorem~\ref{thm:change-of-coordinates} allows one to combine the
Gell--Mann scalarization of Slote--Volberg--Zhang with the Boolean
Bohnenblust--Hille inequality to obtain a bound in an arbitrary tensor-product
operator basis.  We spell this out because it also makes the dependence on the
local dimension transparent.

In \cite[Section~2]{SVZqudit}, Slote--Volberg--Zhang construct, for every
$x\in\{-1,1\}^{K^2-1}$, a density matrix $\rho(x)$ with the following
property.  Write
\(
\mathrm{GM}(K)^\times:=\mathrm{GM}(K)\setminus\{I\},
\)
so that \(|\mathrm{GM}(K)^\times|=K^2-1\), and index the Boolean cube as
\[
x=(x_G)_{G\in\mathrm{GM}(K)^\times}
   \in\{-1,1\}^{\mathrm{GM}(K)^\times}.
\]
By \cite[Lemma~10]{SVZqudit}, for every such $x$
there exists a density matrix \(\rho(x)\) such that
\[
 \tr(G\rho(x))=g_K^{-1}x_G
 \qquad
 (G\in\mathrm{GM}(K)^\times),
\]
where
\[
 g_K:=3\binom K2\sqrt{\frac2K}
      =\frac{3}{\sqrt2}(K-1)\sqrt K.
\]
Consequently, for
$A\in\mathcal H_{\le d,K}(n)$ written in GM coordinates, the product-state
scalarization
\[
 f_A(x^{(1)},\ldots,x^{(n)})
 :=\tr\!\left[A\,\rho(x^{(1)})\otimes\cdots\otimes\rho(x^{(n)})\right]
\]
is a Boolean polynomial of degree at most $d$ and satisfies
$\|f_A\|_\infty\le\|A\|_{\rm op}$.  Moreover, a GM coefficient supported on
$r$ sites becomes a distinct Boolean Fourier coefficient multiplied by
$g_K^{-r}$.  Hence, with $p_d=2d/(d+1)$,
\begin{equation}\label{eq:GM-SVZ-sharp-scalarization}
 \|\widehat A_{\rm GM}\|_{\ell_{p_d}}
 \le
 g_K^d\,\BH^{\le d}_{\{-1,1\}}\,\|A\|_{\rm op}.
\end{equation}
Here we have simply retained the coefficient appearing in the scalarization
itself.  The statement of \cite[Theorem~4]{SVZqudit} uses the slightly coarser
factor $\bigl(\frac32(K^2-K)\bigr)^d$.

Now let $\mathcal B^{\otimes n}$ be any tensor-product operator basis as in
Theorem~\ref{thm:change-of-coordinates}.  Applying that theorem with the GM
basis and then \eqref{eq:GM-SVZ-sharp-scalarization} gives
\[
 \|\widehat A_{\mathcal B}\|_{\ell_{p_d}}
 \le
 \sqrt{K^2-1}\,g_K^d\,
 \BH^{\le d}_{\{-1,1\}}\,\|A\|_{\rm op},
 \qquad A\in\mathcal H_{\le d,K}(n).
\]
Since the Boolean BH constants are subexponential in $d$
(applying ~\cite{DMPBoolean}), the prefactor
$\sqrt{K^2-1}$ disappears after taking $d$th roots.  Thus this route gives
$\beta(K)\le g_K=\frac{3}{\sqrt2}(K-1)\sqrt K$, an exponential base of order
$K^{3/2}$, for every tensor-product orthonormal operator basis.  The scalarization is entirely due to
Slote--Volberg--Zhang; the only observation here is that the coordinate change
makes their reference-basis estimate basis-independent at a cost which is
irrelevant for the $d$th-root asymptotics.  This consequence does not seem to
be stated explicitly in \cite{SVZqudit}.

Of course, $\beta(K)\le g_K$ is not competitive with the direct
basis-free estimate from Theorem~\ref{thm:local-BH}.  We record it only to separate
what follows from the SVZ scalarization plus a change of coordinates from the
stronger intrinsic argument developed above.

Thus the prime/composite distinction encountered in the HW reduction is not
intrinsic to the SVZ scalarization method: the GM route already gives the
optimal interaction exponent \(p_d\) for every \(K\).  The limitation of this
route is instead the weaker dependence on the local dimension \(K\), yielding
an exponential-base bound of order \(K^{3/2}\).

\subsection{The HW route}

A similar observation applies to the Heisenberg--Weyl scalarization of
Slote--Volberg--Zhang~\cite{SVZqudit}.  Here the natural scalar model is a cyclic group rather
than the Boolean cube.  The structure is nevertheless parallel to the GM
route: one scalarizes the reference operator basis, applies a commutative BH
inequality at support degree $d$, and finally uses
Theorem~\ref{thm:change-of-coordinates} to pass to arbitrary tensor-product
operator coordinates.

The only point that needs some care is the choice of HW directions.  Call
$u=(a,b)\in\mathbb Z_K^2$ primitive if
$\gcd(a,b,K)=1$.  Its cyclic span $\langle u\rangle$ is then a free cyclic
submodule of order $K$, and two primitive vectors generate the same submodule
if and only if they differ by multiplication by a unit in
$\mathbb Z_K^\times$.  Choose one primitive generator from each such submodule
and denote the resulting set by $\Gamma_K$.  Then
$$s_K:=|\Gamma_K|=|\mathbb P^1(\mathbb Z_K)|=
K\prod_{p\mid K}(1+1/p)=:\psi(K).$$  For prime $K$ this is $K+1$, exactly the
number of directions used in the prime-dimensional reduction of
\cite{SVZqudit}.

We now use the scalarization mechanism from \cite[Section~3]{SVZqudit}, but
average only over the projective directions $\Gamma_K$.  The reason this is
enough is elementary.  Every nonzero $v\in\mathbb Z_K^2$ belongs to at least
one free cyclic submodule.  This can be checked prime-power by prime-power and
then assembled by the Chinese remainder theorem.  Moreover, if $u$ is
primitive, the HW commutation relation implies that averaging over an
eigenbasis associated with $u$ annihilates $W_v$ unless
$v\in\langle u\rangle$.  If $v=ku$, the surviving term is, up to an irrelevant
phase, the scalar character indexed by $k$.

Thus the same product-state construction as in \cite{SVZqudit} produces, for
every \(A\in\mathcal H_{\le d,K}(n)\), a scalar function \(f_A\) on the
\(n\)-fold product of \(C_K^{s_K}\), that is, on \(C_K^{s_Kn}\), such that
\begin{equation}\label{eq:HW-economical-scalarization}
 \|f_A\|_\infty\le\|A\|_{\rm op},
 \qquad
 \supp\widehat f_A\subset\{m:|\supp m|\le d\},
 \qquad
 \|\widehat A_{\rm HW}\|_{\ell_{p_d}}
 \le s_K^d\|\widehat f_A\|_{\ell_{p_d}}.
\end{equation}
Indeed, an HW tensor supported on $r$ sites gives scalar monomials supported on
those same $r$ sites, each with coefficient $s_K^{-r}$ times the original HW
coefficient.  Distinct HW frequencies do not collide because the scalar labels
record both the chosen projective direction and the frequency inside that
direction.  In composite dimension a nonzero frequency may lie in more than
one free cyclic submodule; this creates additional scalar coefficients and can
only improve the displayed $\ell_{p_d}$ comparison.

This support bookkeeping is the relevant distinction from the composite-$K$
formulation in \cite[Theorem~7]{SVZqudit}.  There the scalarization is measured
by ordinary scalar total degree, which can increase from $d$ to $(K-1)d$, and
the resulting exponent is
$2(K-1)d/((K-1)d+1)$.  If instead one keeps track only of the number of active
scalar coordinates, the support remains at most $d$ for every $K$.  In this
sense the scalarization argument of \cite{SVZqudit} itself works uniformly in
prime and composite dimensions; the stronger support-degree consequence does
not seem to have been recorded there.

To finish the argument, we use the support-sensitive commutative
Bohnenblust--Hille inequality on \(C_K^N\) proved in \cite{DGMMM-support}.
This already gives the optimal exponent \(p_d=2d/(d+1)\) for every \(K\).
Let \(\BH^{\rm supp}_{d,K}\) denote its optimal constant.  The recent result
of Pellegrino--Raposo \cite[Theorem~A]{PR} improves the growth of these
constants from exponential to subexponential; more precisely,
\begin{equation}\label{eq:PR-support-subexp}
 \BH^{\rm supp}_{d,K}
 \le \exp\!\left(
 \kappa_K\sqrt{d\log d}
 +O_K\!\left(\sqrt{\frac d{\log d}}\log\log d\right)\right),
\end{equation}
where $\kappa_2=2$ and
$\kappa_K=\sqrt{2K\log(K-1)/(K-2)}$ for $K\ge3$.
Combining \eqref{eq:HW-economical-scalarization} with this inequality gives
\[
 \|\widehat A_{\rm HW}\|_{\ell_{p_d}}
 \le
 s_K^d\BH^{\rm supp}_{d,K}\|A\|_{\rm op}.
\]
Finally, Theorem~\ref{thm:change-of-coordinates} yields, for every
normalized Hilbert--Schmidt orthonormal tensor-product operator basis
$\mathcal B^{\otimes n}$,
\[
 \|\widehat A_{\mathcal B}\|_{\ell_{p_d}}
 \le
 \sqrt{K^2-1}\,s_K^d\BH^{\rm supp}_{d,K}\|A\|_{\rm op},
 \qquad A\in\mathcal H_{\le d,K}(n).
\]
Since the commutative factor in \eqref{eq:PR-support-subexp} is
subexponential in $d$, the $d$th-root consequence is
$\beta(K)\le s_K=\psi(K)$.  The classical maximal-order formula
$\limsup_{K\to\infty}s_K/(K\log\log K)=6e^\gamma/\pi^2$ shows in particular
that $s_K\le C K\log\log(K+3)$.

This gives a convenient way to read the HW scalarization of
Slote--Volberg--Zhang together with the change of coordinates: their
reference-basis procedure, reorganized by projective HW directions and coupled
with the support-sensitive scalar BH inequality, gives the optimal BH exponent
$p_d$ for every local dimension, including composite $K$, and for every
tensor-product orthonormal operator basis.  The projective reorganization is
what improves the scalarization count from a quadratic-size family of
coprime pairs to $s_K=\psi(K)$, of maximal order $K\log\log K$.

As in the GM discussion, this route should be viewed as clarifying what can
be extracted from the SVZ scalarization rather than as replacing the direct
approach of Section~\ref{sec:BH-theory}.  In particular, the bound on
\(\beta(K)\) obtained in Theorem~\ref{thm:local-BH} remains substantially
stronger.



\section{Coefficient geometry and quantum applications}
\label{sec:applications}
The Bohnenblust--Hille estimates obtained above control the coefficient
geometry of local operators beyond the critical \(\ell_{p_d}\)-norm.
A natural next quantity is the coefficient one-norm
\(\|\widehat A_{\mathcal B}\|_1\), which measures the size of a coordinate
representation relative to the operator norm.  We first determine its
sharp scale, together with that of the unconditional basis constant, on
the exact-support spaces.

This coefficient geometry leads to two natural quantum applications.
We formulate them mainly in functional-analytic language, both to keep
the connection with the preceding sections transparent and because their
essential content is already operator-theoretic.  The first is a
nonlinear approximation statement: local operators admit sparse
tensor-product approximations, and hence sparse generalized-Pauli
approximations in Heisenberg--Weyl coordinates.  Such representations are
natural for compressed descriptions and learning problems, where one
seeks to retain only the significant coefficients.

The second is a unitary compression statement.  A Hamiltonian is
Hermitian but generally not unitary, whereas coherent quantum operations
are represented by unitaries.  In unitary tensor-product coordinates the
coefficient one-norm is precisely the LCU normalization, and the
PREPARE--SELECT construction realizes the normalized Hamiltonian as a
compression of a unitary operator.  Thus the same coefficient geometry
leads naturally from sparsification to block encodings and qubitization.

\subsection{Coefficient one-norm normalization and unconditionality}
\label{sec:coefficient-normalization}
Let
\(
\mathcal B=\{I,B_1,\ldots,B_{K^2-1}\}
\)
be a normalized Hilbert--Schmidt orthonormal local operator basis, and let
\(\mathcal B^{\otimes n}\) be the associated tensor-product basis.  For
\(\beta=(\beta_1,\ldots,\beta_n)\), write
\[
B_\beta=B_{\beta_1}\otimes\cdots\otimes B_{\beta_n},
\qquad
\mathcal B_{d,n}
=
\{B_\beta:|\supp\beta|=d\}.
\]
Thus
\[
M_{d,K,n}:=|\mathcal B_{d,n}|
=
\binom nd(K^2-1)^d.
\]
For \(A\in\mathcal H_{d,K}(n)\), let
\(\widehat A_{\mathcal B}\) denote its coefficient vector in
\(\mathcal B^{\otimes n}\).  We define the coefficient one-norm
normalization constant by
\[
\Lambda_{\mathcal B}(\mathcal H_{d,K}(n))
:=
\sup_{A\ne0}
\frac{\|\widehat A_{\mathcal B}\|_1}{\|A\|_{\rm op}}.
\]
Equivalently, this is the Sidon constant of the tensor system
\(\mathcal B_{d,n}\).
The unconditional basis constant of the same tensor system is
\[
\chi_{\mathcal B}(\mathcal H_{d,K}(n))
:=
\sup_{\substack{A\ne0\\ |\varepsilon_\beta|\le1}}
\frac{
\left\|
\sum_{|\supp\beta|=d}
\varepsilon_\beta
\widehat A_{\mathcal B}(\beta)B_\beta
\right\|_{\rm op}}
{\|A\|_{\rm op}}.
\]
Later, when \(\mathcal B=\mathcal U\) is unitary and \(H\) is an individual
Hamiltonian, we shall use the lower-case notation
\(\lambda_{\mathcal U}(H)=\|\widehat H_{\mathcal U}\|_1\) for its LCU
normalization.
The preceding Bohnenblust--Hille estimates determine the \(n\)-dependence
of both quantities sharply.

\begin{theorem}[Coefficient normalization and unconditionality scale]
\label{thm:l1-chi-scale}
Fix \(K\ge2\).  There is a constant \(c(K)\ge1\) such that, for every
\(1\le d\le n\), every normalized Hilbert--Schmidt orthonormal local
operator basis \(\mathcal B\), and each
\[
Q_{\mathcal B}
\in
\left\{
\Lambda_{\mathcal B}(\mathcal H_{d,K}(n)),
\chi_{\mathcal B}(\mathcal H_{d,K}(n))
\right\},
\]
one has
\begin{equation}\label{eq:l1-chi-scale}
c(K)^{-d}
\left(\frac nd\right)^{(d-1)/2}
\le
Q_{\mathcal B}
\le
c(K)^d
\left(\frac nd\right)^{(d-1)/2}.
\end{equation}
Equivalently,
\[
\Lambda_{\mathcal B}(\mathcal H_{d,K}(n))
\asymp_{c(K)^d}
\chi_{\mathcal B}(\mathcal H_{d,K}(n))
\asymp_{c(K)^d}
\left(\frac nd\right)^{(d-1)/2},
\]
uniformly in \(n,d\) and in the choice of \(\mathcal B\).
\end{theorem}

\begin{proof}
\medskip
\noindent\emph{Upper bound.}
By Theorem~\ref{thm:local-BH},
\[
\|\widehat A_{\mathcal B}\|_{p_d}
\le
c_{\rm BH}(K)^d\|A\|_{\rm op},
\qquad
p_d=\frac{2d}{d+1}.
\]
Since \(\widehat A_{\mathcal B}\) has at most
\(M_{d,K,n}=\binom nd(K^2-1)^d\) nonzero coordinates, H\"older's inequality
gives
\[
\|\widehat A_{\mathcal B}\|_1
\le
M_{d,K,n}^{1-1/p_d}
\|\widehat A_{\mathcal B}\|_{p_d}
=
M_{d,K,n}^{(d-1)/(2d)}
\|\widehat A_{\mathcal B}\|_{p_d}.
\]
Using \(\binom nd\le(en/d)^d\), we obtain
\[
\Lambda_{\mathcal B}(\mathcal H_{d,K}(n))
\le
\bigl(c_{\rm BH}(K)\sqrt{e(K^2-1)}\bigr)^d
\left(\frac nd\right)^{(d-1)/2}.
\]
Moreover, normalized Hilbert--Schmidt orthonormality implies
\(\|B_\mu\|_{\rm op}\le\sqrt K\), and therefore
\(\|B_\beta\|_{\rm op}\le K^{d/2}\) whenever
\(|\supp\beta|=d\).  Hence, for every family
\((\varepsilon_\beta)\) with \(|\varepsilon_\beta|\le1\),
\[
\left\|
\sum_{|\supp\beta|=d}
\varepsilon_\beta
\widehat A_{\mathcal B}(\beta)B_\beta
\right\|_{\rm op}
\le
K^{d/2}\|\widehat A_{\mathcal B}\|_1.
\]
Thus
\[
\chi_{\mathcal B}(\mathcal H_{d,K}(n))
\le
K^{d/2}
\Lambda_{\mathcal B}(\mathcal H_{d,K}(n)),
\]
which proves the required upper bound after enlarging the constant depending
only on \(K\).

\medskip
\noindent\emph{Lower bound.}
We use the commutative spherical corner from
Lemma~\ref{lem:spherical-corners}, transported to coordinates adapted to
\(\mathcal B\).
Complete the diagonal orthonormal family
\(\{I,D_1,\ldots,D_{K-1}\}\) from \eqref{eq:Dm-definition} to a normalized
Hilbert--Schmidt orthonormal basis
\[
\mathcal D=\{I,D_1,\ldots,D_{K^2-1}\},
\]
and let
\[
R_{\mathcal B}:M_K(\mathbb C)\longrightarrow M_K(\mathbb C)
\]
be the Hilbert--Schmidt unitary determined by
\[
R_{\mathcal B}(I)=I,
\qquad
R_{\mathcal B}(D_\mu)=B_\mu
\quad
(1\le\mu\le K^2-1).
\]
We claim that, on exact support,
\begin{equation}\label{eq:uniform-exact-support-change}
\|R_{\mathcal B,d,n}\|,
\quad
\|R_{\mathcal B,d,n}^{-1}\|
\le
A_K^d,
\qquad
R_{\mathcal B,d,n}
:=
R_{\mathcal B}^{\otimes n}\big|_{\mathcal H_{d,K}(n)},
\end{equation}
where one may take
\[
A_K=2(1+\sqrt2)K^2.
\]
To prove this, let
\[
\tau_K(A)=K^{-1}\tr A,
\qquad
E(A)=\tau_K(A)I,
\qquad
S_{\mathcal B}=R_{\mathcal B}-E.
\]
Then \(S_{\mathcal B}\) vanishes on \(\mathbb CI\), is Hilbert--Schmidt
unitary on \(M_K^0\), and has Hilbert--Schmidt operator norm one.  For a
linear map \(T\), put
\[
T^\sharp(A)=T(A^*)^*,
\qquad
\Delta_1=\frac{S_{\mathcal B}+S_{\mathcal B}^\sharp}{2},
\qquad
\Delta_2=\frac{S_{\mathcal B}-S_{\mathcal B}^\sharp}{2i}.
\]
Both \(\Delta_1\) and \(\Delta_2\) are hermiticity preserving, vanish at
\(I\), and have Hilbert--Schmidt operator norm at most one.
Let
\[
C_T=\sum_{r,s=1}^K E_{rs}\otimes T(E_{rs})
\]
be the Choi matrix of \(T\).  With \(\|\cdot\|_F\) denoting the
unnormalized Frobenius norm,
\[
\|C_{\Delta_j}\|_{\rm op}
\le
\|C_{\Delta_j}\|_F
=
\left(
\sum_{r,s=1}^K
\|\Delta_j(E_{rs})\|_F^2
\right)^{1/2}
\le K,
\qquad j=1,2.
\]
Since \(C_E=K^{-1}I_{K^2}\), the map
\[
\Phi_{s,t}:=E+s\Delta_1+t\Delta_2
\]
is unital completely positive whenever
\[
|s|,|t|\le\delta_K,
\qquad
\delta_K:=\frac1{2K^2}.
\]
Consequently,
\[
\|\Phi_{s,t}^{\otimes n}\|_{\rm op\to op}=1
\]
throughout this square.
For \(A\in\mathcal H_{d,K}(n)\), the polynomial
\[
F_A(s,t):=\Phi_{s,t}^{\otimes n}(A)
\]
is homogeneous of degree \(d\).  Hence, for \(-1\le t\le1\),
\[
\delta_K^d\|F_A(1,t)\|_{\rm op}
=
\|F_A(\delta_K,\delta_Kt)\|_{\rm op}
\le
\|A\|_{\rm op}.
\]
Applying the scalar Bernstein--Walsh inequality to
\(\varphi(\delta_K^dF_A(1,z))\), with \(\|\varphi\|\le1\), and evaluating
at \(z=i\), yields
\[
\delta_K^d\|F_A(1,i)\|_{\rm op}
\le
(1+\sqrt2)^d\|A\|_{\rm op}.
\]
Since \(\Phi_{1,i}=R_{\mathcal B}\), this proves
\[
\|R_{\mathcal B,d,n}\|
\le
[2(1+\sqrt2)K^2]^d.
\]
The inverse coordinate change is of the same form, so the same estimate
holds for \(R_{\mathcal B,d,n}^{-1}\), proving
\eqref{eq:uniform-exact-support-change}.
Let now
\[
J:B_d(C_K^n)\longrightarrow\mathcal H_{d,K}(n)
\]
be the isometric embedding from Lemma~\ref{lem:spherical-corners}, and put
\[
J_{\mathcal B}:=R_{\mathcal B,d,n}\circ J.
\]
For every character \(\chi_m\) with \(|\supp(m)|=d\),
\[
J_{\mathcal B}(\chi_m)
=
B_{m_1}\otimes\cdots\otimes B_{m_n},
\qquad B_0=I.
\]
Thus the range of \(J_{\mathcal B}\) is spanned by a genuine subfamily of
\(\mathcal B_{d,n}\), and
\[
\|J_{\mathcal B}\|,
\quad
\|J_{\mathcal B}^{-1}\|
\le
A_K^d.
\]
By \cite[Theorem~3.4]{DGMMM-support}, the coefficient one-norm and
unconditional constants of the character basis of \(B_d(C_K^n)\) are
bounded below by
\[
a(K)^d
\left(\frac nd\right)^{(d-1)/2}
\]
for some \(a(K)>0\).
If \(A=J_{\mathcal B}f\), then the \(\mathcal B\)-coefficient one-norm of
\(A\) equals the character coefficient one-norm of \(f\), while
\[
\|A\|_{\rm op}\le A_K^d\|f\|_\infty.
\]
Therefore
\[
\Lambda_{\mathcal B}(\mathcal H_{d,K}(n))
\ge
A_K^{-d}a(K)^d
\left(\frac nd\right)^{(d-1)/2}.
\]
For unconditionality, let \(T_\varepsilon\) be a diagonal multiplier on
the character basis of \(B_d(C_K^n)\), and extend the same multipliers by
zero to the remaining elements of \(\mathcal B_{d,n}\), obtaining a
multiplier \(M_\varepsilon\) on \(\mathcal H_{d,K}(n)\).  On the range of
\(J_{\mathcal B}\),
\[
M_\varepsilon J_{\mathcal B}
=
J_{\mathcal B}T_\varepsilon,
\]
and hence
\[
\|T_\varepsilon\|
\le
\|J_{\mathcal B}^{-1}\|
\|M_\varepsilon\|
\|J_{\mathcal B}\|
\le
A_K^{2d}\|M_\varepsilon\|.
\]
Taking suprema gives
\[
\chi_{\mathcal B}(\mathcal H_{d,K}(n))
\ge
A_K^{-2d}a(K)^d
\left(\frac nd\right)^{(d-1)/2}.
\]
Combining the upper and lower estimates and enlarging the constant depending
only on \(K\) proves \eqref{eq:l1-chi-scale}.
\end{proof}

\subsection{Basis-independent coefficient compressibility}
\label{subsec:coefficient-compressibility}
The coefficient geometry developed above has a complementary sparsity
consequence.  While the preceding subsection determines the sharp
\(\ell_1\)-scale, the Bohnenblust--Hille estimate controls the decay of
individual coefficients and hence yields quantitative sparse
approximation in every tensor-product Hilbert--Schmidt orthonormal basis.

Fix a normalized Hilbert--Schmidt orthonormal local operator basis
\(
\mathcal B=\{I,B_1,\ldots,B_{K^2-1}\}
\)
and let \(H\in\mathcal H_{\le d,K}(n)\).  Rearrange the moduli of the
coefficients of \(H\) in \(\mathcal B^{\otimes n}\) in nonincreasing
order,
\[
h^*_{1,\mathcal B}\ge h^*_{2,\mathcal B}\ge\cdots\ge0.
\]
If \(C_{\le d,K}\) denotes the dimension-free constant in the
Bohnenblust--Hille inequality for the full \(d\)-local space, then
\[
h^*_{j,\mathcal B}
\le
C_{\le d,K}\|H\|_{\op}\,j^{-1/p_d},
\qquad
p_d=\frac{2d}{d+1}.
\]
Indeed, this is the standard weak-\(\ell_{p_d}\) consequence of the
estimate
\[
\|\widehat H_{\mathcal B}\|_{\ell_{p_d}}
\le
C_{\le d,K}\|H\|_{\op}.
\]
Consequently, for every \(\tau>0\),
\[
\#\bigl\{
\beta:
|\widehat H_{\mathcal B}(\beta)|\ge\tau
\bigr\}
\le
\left(
\frac{C_{\le d,K}\|H\|_{\op}}{\tau}
\right)^{p_d}.
\]
Thus, for fixed \(K\) and \(d\), the number of coefficients above a fixed
operator-norm-relative threshold is independent of the number \(n\) of
qudits and of the chosen local orthonormal basis.

Let \(H_{s,\mathcal B}\) be obtained by retaining the \(s\) largest
coefficients of \(H\) in \(\mathcal B^{\otimes n}\).  With the normalized
Hilbert--Schmidt norm
\[
\|A\|_2
=
\bigl(K^{-n}\tr(A^*A)\bigr)^{1/2},
\]
orthonormality gives
\[
\|H-H_{s,\mathcal B}\|_2^2
=
\sum_{j>s}(h^*_{j,\mathcal B})^2.
\]
Since
\[
\sum_{j>s}j^{-(d+1)/d}
\le
d\,s^{-1/d},
\]
we obtain
\[
\|H-H_{s,\mathcal B}\|_2
\le
\sqrt d\,
C_{\le d,K}\|H\|_{\op}\,
s^{-1/(2d)}.
\]
Hence bounded \(d\)-local operators admit dimension-free sparse
approximations in every tensor-product Hilbert--Schmidt orthonormal
basis.  In Heisenberg--Weyl coordinates this is precisely sparse
generalized-Pauli approximation.

\subsection{LCU normalization and unitary compression}
\label{subsec:lcu-block-encoding}
We now pass to unitary tensor-product coordinates, where the coefficient
one-norm becomes the normalization parameter of the canonical
linear-combination-of-unitaries (LCU) representation.

We specialize to a normalized Hilbert--Schmidt orthonormal local
basis
\[
\mathcal U
=
\{U_0=I,U_1,\ldots,U_{K^2-1}\}
\]
consisting of unitary matrices.  Its tensor powers are again unitary.
For
\[
\beta=(\beta_1,\ldots,\beta_n)
\in
\{0,\ldots,K^2-1\}^n,
\]
write
\[
U_\beta
:=
U_{\beta_1}\otimes\cdots\otimes U_{\beta_n}.
\]
Every \(H\in M_K(\mathbb C)^{\otimes n}\) has a unique expansion
\[
H
=
\sum_\beta h_\beta U_\beta,
\]
where
\[
h_\beta
=
\langle H,U_\beta\rangle_2
=
K^{-n}\tr(HU_\beta^*).
\]
We denote its coefficient vector by
\[
\widehat H_{\mathcal U}
:=
(h_\beta)_\beta.
\]
Assume now that \(H=H^*\) and \(H\in\mathcal H_{\le d,K}(n)\).
For every nonzero coefficient \(h_\beta\), absorb its phase into the
corresponding unitary and put
\[
V_\beta
:=
\frac{h_\beta}{|h_\beta|}U_\beta.
\]
Then \(V_\beta\) is unitary and
\[
H
=
\sum_{\beta\in\Gamma}|h_\beta|V_\beta,
\qquad
\Gamma
:=
\{\beta:h_\beta\neq0\}.
\]
The associated LCU normalization is therefore
\[
\lambda_{\mathcal U}(H)
:=
\sum_{\beta\in\Gamma}|h_\beta|
=
\|\widehat H_{\mathcal U}\|_1.
\]
We next describe the functional-analytic content of the standard
PREPARE--SELECT construction.  Let \(\mathcal H\) denote the system
Hilbert space and consider the enlarged Hilbert space
\[
\mathcal K
:=
\ell_2(\Gamma)\otimes\mathcal H.
\]
Let \((e_\beta)_{\beta\in\Gamma}\) be the canonical orthonormal basis of
\(\ell_2(\Gamma)\), and define
\[
\xi
:=
\sum_{\beta\in\Gamma}
\sqrt{\frac{|h_\beta|}{\lambda_{\mathcal U}(H)}}\,e_\beta.
\]
Since
\[
\sum_{\beta\in\Gamma}
\frac{|h_\beta|}{\lambda_{\mathcal U}(H)}
=
1,
\]
the vector \(\xi\) has norm one.
Define the block-diagonal operator
\[
S
:=
\sum_{\beta\in\Gamma}
|e_\beta\rangle\langle e_\beta|
\otimes V_\beta
=
\bigoplus_{\beta\in\Gamma}V_\beta
\qquad\text{on }\mathcal K.
\]
Since every \(V_\beta\) is unitary, \(S\) is unitary.  Define further
the isometric embedding
\[
J:\mathcal H\longrightarrow\mathcal K,
\qquad
Jx:=\xi\otimes x.
\]
Then
\[
J^*J=I_{\mathcal H},
\]
and a direct computation gives
\[
\begin{aligned}
J^*SJ
=
\sum_{\beta\in\Gamma}
|\langle e_\beta,\xi\rangle|^2V_\beta =
\sum_{\beta\in\Gamma}
\frac{|h_\beta|}{\lambda_{\mathcal U}(H)}V_\beta
=
\frac{H}{\lambda_{\mathcal U}(H)}.
\end{aligned}
\]
Thus \(H/\lambda_{\mathcal U}(H)\) is the compression of the unitary
operator \(S\) to the isometrically embedded copy \(J\mathcal H\) of
the original system space.

This identity is the functional-analytic core of PREPARE--SELECT.
To recover the usual formulation, fix a unit vector
\(e_0\in\ell_2(\Gamma)\) and choose a unitary
\[
P:\ell_2(\Gamma)\longrightarrow\ell_2(\Gamma) \quad \text{such that $Pe_0=\xi.$}
\]
In quantum-information terminology, \(P\) is the
\textsc{Prepare} operator, while \(S\) is the \textsc{Select} operator.
If
\[
J_0:\mathcal H\longrightarrow\mathcal K,
\qquad
J_0x=e_0\otimes x,
\]
and
\[
W
:=
(P^*\otimes I_{\mathcal H})
S
(P\otimes I_{\mathcal H}),
\]
then \(W\) is unitary and
\begin{equation*}\label{eq:block-encoding}
J_0^*WJ_0
=
\frac{H}{\lambda_{\mathcal U}(H)}.
\end{equation*}
Equivalently,
\[
(\langle e_0|\otimes I_{\mathcal H})
W
(|e_0\rangle\otimes I_{\mathcal H})
=
\frac{H}{\lambda_{\mathcal U}(H)}.
\]
With respect to the orthogonal decomposition
\[
\mathcal K
=
J_0\mathcal H
\oplus
(J_0\mathcal H)^\perp,
\]
the unitary \(W\) therefore has the block form
\[
W
=
\begin{pmatrix}
H/\lambda_{\mathcal U}(H) & *\\
* & *
\end{pmatrix}.
\]
In this sense \(W\) is a block encoding of
\(H/\lambda_{\mathcal U}(H)\).

The significance of this representation for Hamiltonian simulation is
that quantum circuits provide direct access to unitary transformations,
whereas a Hamiltonian \(H\) is in general not unitary.  This compression embeds the normalized Hamiltonian into a
unitary operator on a larger Hilbert space.
Qubitization and quantum
signal-processing methods can then be applied to this unitary encoding
to implement functions of \(H\), in particular the time evolution
\(e^{-itH}\).

The normalization parameter entering this construction is exactly
\(\lambda_{\mathcal U}(H)=\|\widehat H_{\mathcal U}\|_1\).  Hence the
coefficient estimates from Section~\ref{sec:coefficient-normalization}
translate directly into bounds for the normalization of the canonical
unitary-basis LCU encoding.

\begin{corollary}[Normalization scale for canonical unitary-basis LCU encodings]
\label{cor:lcu-normalization-scale}
For every fixed \(K\ge2\) there are constants \(c_+(K)\ge1\) and
\(c_-(K)>0\) such that, uniformly over all normalized
Hilbert--Schmidt orthonormal unitary local bases \(\mathcal U\), every
Hermitian \(H\in\mathcal H_{\le d,K}(n)\) satisfies
\[
\lambda_{\mathcal U}(H)
\le
c_+(K)^d
\left(\frac nd\right)^{(d-1)/2}
\|H\|_{\op}.
\]
Moreover, the power \((d-1)/2\) is sharp already on exact support:
for every \(1\le d\le n\) and every such \(\mathcal U\), there exists a
Hermitian \(H\in\mathcal H_{d,K}(n)\) such that
\[
\lambda_{\mathcal U}(H)
\ge
c_-(K)^d
\left(\frac nd\right)^{(d-1)/2}
\|H\|_{\op}.
\]
The constants may be chosen uniformly over \(\mathcal U\).
\end{corollary}

\begin{proof}
For exact support this is precisely the coefficient
$\ell_1$-normalization estimate of Theorem~\ref{thm:l1-chi-scale},
because of the identity
\[
\lambda_{\mathcal U}(H)
=
\|\widehat H_{\mathcal U}\|_1.
\]
For the full \(d\)-local space, combine the Bohnenblust--Hille estimate
from Theorem~\ref{thm:local-BH-leq-d} with H\"older's inequality over at
most
\[
\sum_{r=0}^d
\binom nr(K^2-1)^r
\le
\left(c(K)\frac nd\right)^d
\]
coefficients.
For the Hermitian lower bound, choose an almost extremal
\(A\in\mathcal H_{d,K}(n)\) for the coefficient
\(\ell_1\)-normalization and write
\(
A=H_1+iH_2,
\)
where \(H_1,H_2\) are Hermitian.  Then
\(
\|H_j\|_{\op}\le\|A\|_{\op},
\)
and, by linearity of the coefficient map,
\[
\|\widehat A_{\mathcal U}\|_1
\le
\|\widehat H_{1,\mathcal U}\|_1
+
\|\widehat H_{2,\mathcal U}\|_1.
\]
Thus one of the two Hermitian parts retains at least one half of the
extremal coefficient ratio.
\end{proof}

Given unit-cost oracle access to \textsc{Prepare} and \textsc{Select},
the qubitization framework of Low and Chuang~\cite[Corollary~16]{LowChuang} implies
that an LCU representation
\[
H=\sum_j \alpha_jV_j,
\qquad
\alpha:=\sum_j|\alpha_j|,
\]
allows one to simulate \(e^{-itH}\) to spectral-norm error
\(\varepsilon\) using
\[
O\left(
\alpha t+\log\frac1\varepsilon
\right)
\]
oracle queries.  Applying this with
\(\alpha=\lambda_{\mathcal U}(H)\) yields the following consequence.

\begin{corollary}[Qubitization from a unitary tensor-product basis]
\label{cor:qubitization-unitary-basis}
Assume unit-cost query access to the \textsc{Prepare} and
\textsc{Select} operators associated with the expansion of a Hermitian
\(H\in\mathcal H_{\le d,K}(n)\) in a normalized Hilbert--Schmidt
orthonormal unitary local basis \(\mathcal U\).  Then \(e^{-itH}\) can
be simulated to spectral-norm error \(\varepsilon\) using
\[
O\left(
c(K)^d
\left(\frac nd\right)^{(d-1)/2}
\|H\|_{\op}t
+
\log\frac1\varepsilon
\right)
\]
queries to the corresponding LCU oracles.
For \(\mathcal U\) equal to the Heisenberg--Weyl basis, this is the
canonical generalized-Pauli PREPARE--SELECT encoding.
\end{corollary}

We emphasize that 
Corollary~\ref{cor:qubitization-unitary-basis} is a statement
about the normalization and query complexity of this particular
unitary-basis LCU representation.  It is not a lower bound for arbitrary
Hamiltonian-simulation algorithms or for arbitrary block encodings.
Additional structure of \(H\), a different decomposition, or a more
efficient access model may reduce the complexity.  Moreover, the gate
complexity includes the cost of implementing \textsc{Prepare} and
\textsc{Select} themselves.  What is sharp here is the worst-case
\(n\)-dependence of the coefficient normalization in the chosen
unitary tensor-product coordinates.

Finally, let \(\mathcal B=\{I,B_1,\ldots,B_{K^2-1}\}\) be an arbitrary,
not necessarily unitary, normalized Hilbert--Schmidt orthonormal local
basis.  Then
\[
\|B_\mu\|_{\op}\le\sqrt K.
\]
Hence \(B_\mu/\sqrt K\) is a contraction and, in finite dimensions,
belongs to the convex hull of the unitary group.  Choosing such a
unitary decomposition for every local basis element and tensorizing
over the active sites shows that an expansion
\[
H=\sum_\beta h_\beta B_\beta
\]
with support at most \(d\) can be converted into an LCU of tensor
products of local unitaries with normalization at most
\[
\sum_\beta
K^{|\supp\beta|/2}|h_\beta|
\le
K^{d/2}
\|\widehat H_{\mathcal B}\|_1.
\]
Thus the basis-independent coefficient estimates above also yield LCU
upper bounds for arbitrary local orthonormal bases, at the price of the
additional factor \(K^{d/2}\).  This last conversion is an existence
statement; its implementation cost depends on the chosen local unitary
decompositions.  Unitary bases, and in particular the
Heisenberg--Weyl basis, are distinguished because no such conversion
loss occurs.

\section*{Note added}
While this manuscript was being finalized, the authors became aware of the
recent preprint by Slote and Volberg \cite{SloteVolberg2026}, which proposes
polynomial Bohnenblust--Hille bounds in the commutative setting of products
of cyclic groups.
 We emphasize, however, that the noncommutative framework considered here is fundamentally different: in this case, the dependence on the interaction order cannot, in general, be subexponential.

\section*{Declaration on the use of generative AI}

During the preparation of this manuscript, the authors used ChatGPT (GPT-5.6 Sol, OpenAI) for language editing, revision of the manuscript, assistance with bibliographic preparation, and informal discussion and exploration of mathematical ideas. The authors take full responsibility for the content of the paper.

\section*{Acknowledgements}

The second author is deeply grateful to Andreas Defant for the opportunity to stay at his home in Oldenburg during a visit while working on this project, and would also like to warmly thank Ute for her wonderful hospitality and kindness.

\vspace{1.5em}
\noindent\textsc{Institut f\"ur Mathematik, Carl von Ossietzky Universit\"at,
26111 Oldenburg, Germany}\\
\textit{Email address:} \texttt{defant@mathematik.uni-oldenburg.de}

\medskip
\noindent\textsc{Universidad Torcuato Di Tella. Departamento de Matem\'atica y
Estad\'istica. IMAS--CONICET, Av. Figueroa Alcorta 7350 (1428), Buenos Aires,
Argentina}\\
\textit{Email address:} \texttt{daniel.galicer@utdt.edu}

\end{document}